\documentclass{amsart}

\usepackage{amsmath, amsthm}
\usepackage{amssymb}
\usepackage{amscd}
\usepackage{latexsym}
\usepackage[french]{babel}
\usepackage[utf8]{inputenc}

\usepackage[shortlabels]{enumitem}
\usepackage{xr}

\input xy
\xyoption{all}

\usepackage{xypic} 
\usepackage{graphicx}
\usepackage{changebar,color}

\usepackage{cyrillic}

\usepackage{tikz}

\usepackage{mathrsfs}

\usepackage{tikz-cd}

\usepackage{array}
\usepackage{xy}
\xyoption{all}
\usepackage{mathtools}

\usepackage{imakeidx} 
\usepackage{appendix}
\usepackage{tikz-cd}
\usepackage{verbatim}
\usepackage{enumitem}
\usepackage{multicol}

\usepackage[backref=page]{hyperref}
\usepackage{cleveref}

 \hypersetup{colorlinks=true,linkcolor=blue,anchorcolor=blue,citecolor=blue}

\usepackage{adjustbox}

\usepackage{mathtools}

\usepackage{color}

\newtheorem{thm}{Th\'eor\`eme}[section]
\newtheorem{prop}[thm]{Proposition}
\newtheorem{lemma}[thm]{Lemme}
\newtheorem{cor}[thm]{Corollaire}

\newtheorem{situation}[thm]{Situation}

\theoremstyle{definition}

\theoremstyle{remark}

\newtheorem{rmk}[thm]{Remarque}

\numberwithin{equation}{section}

\newcommand{\Gal}{{\rm Gal}}
\newcommand{\Br}{{\rm Br}}
\newcommand{\Q}{\mathbb Q}
\newcommand{\R}{\mathbb R}

\newcommand{\C}{\mathbb C}
\newcommand{\Z}{\mathbb Z}

\renewcommand{\P}{\mathbb P}
\newcommand{\Spec}{\operatorname{Spec}}

\newcommand{\A}{\mathbb A}

\renewcommand{\phi}{\varphi}

\newcommand{\cyrit}{\fontencoding{OT2}\selectfont\textcyrit}

\def\Pic{{\rm Pic}}
\def\Ker{{\rm Ker}}
\def\Im{{\rm Im}}
\def\Div{{\rm Div}}

\title[Fibrations en   quadriques r\'eelles]{Certaines fibrations en surfaces quadriques   r\'eelles}
\author{Jean-Louis Colliot-Th\'el\`ene et Alena Pirutka}

\address{Universit\'e Paris-Saclay, CNRS, Laboratoire de math\'ematiques d'Orsay, 91405, Orsay, France.}
\email{jean-louis.colliot-thelene@universite-paris-saclay.fr}

\address{Courant Institute of Mathematical Sciences, New York University, New York, 10012, U.S.A.}\email{pirutka@cims.nyu.edu}

\date{article soumis le 20 juin 2024;  rapport  du 28 juillet 2025 ; version r\'evis\'ee soumise le 4 septembre 2025; accept\'e le 2 f\'evrier 2026}

\begin{document}

	\maketitle
	
\begin{abstract} 
\vspace{-1cm}  

 Soit $X/\R$ une vari\'et\'e projective et lisse sur le corps des r\'eels,
g\'eom\'e\-tri\-quement rationnelle,
 telle que l'espace $X(\R)$ est non vide et connexe.
La vari\'et\'e $X$ est-elle rationnelle sur $\R$ ? Est-elle stablement rationnelle ?
Est-elle r\'etracti\-lement rationnelle ?  Est-elle universellement $CH_{0}$-triviale ?
 
Dans cet article on consid\`ere le cas des solides $X$  fibr\'es en surfaces quadriques sur la droite projective.
Nous donnons un exemple pour lequel $X$ a un groupe de Brauer non trivial et donc
la r\'eponse aux questions ci-dessus est n\'egative.

Lorsque toutes les fibres g\'eom\'etriques  sont int\`egres,
le groupe de Brauer  de $X$ est trivial.    Nous examinons une classe de telles fibrations pour lesquelles
de plus la m\'ethode des torseurs de jacobiennes interm\'ediaires pour infirmer la rationalit\'e
ne donne pas d'obstruction.
 
Pour $X$ dans cette classe, 
en utilisant  le troisi\`eme groupe de cohomologie non ramifi\'ee de $X$
sur tout corps contenant $\R$, nous exhibons un invariant permettant de
d\'ecider si  $X$ est  universellement $CH_{0}$-triviale. Par deux m\'ethodes 
ind\'ependantes,
l'une
 utilisant
  les sommes de carr\'es, l'autre la multiplication complexe, 
nous obtenons ainsi de nombreuses  classes
 de  telles fibrations  pour lesquelles  $X$ est  universellement $CH_{0}$-triviale.
 
 \bigskip

\noindent \textsc{Abstract.} We consider the question whether a real threefold $X$ fibred into quadric surfaces
 over the real projective line is stably rational (over $\R$) if the topological space $X(\R)$ is connected. 
 We give a counterexample. When all geometric fibres are irreducible, the question is open.
 We investigate a family of such fibrations for which the intermediate jacobian technique 
 gives no obstruction to rationality.
For $X$ in this class, we use the third unramified cohomology group of
$X$ over any field containing $\R$ to produce an invariant which
characterizes decomposition of the diagonal for $X$ over $\R$.
We then produce two independent methods (sums of squares and complex multiplication) which in many cases but not all
 enable us  to prove decomposition of the diagonal.
 \end{abstract}

 \section{Introduction}
 
  \subsection{Rappels}

Soit $k$ un corps.
Une $k$-vari\'et\'e est un $k$-sch\'ema s\'epar\'e de type fini. Une $k$-vari\'et\'e int\`egre
est dite {\it $k$-rationnelle} si elle est $k$-birationnelle \`a un espace projectif $\P^n_{k}$.
Une $k$-vari\'et\'e int\`egre $X$ est dite {\it stablement $k$-rationnelle} s'il existe des espaces projectifs
$\P^n_{k}$ et $\P^m_{k}$ tels que $X \times_{k}\P^n_{k}$ est $k$-birationnel \`a $\P^m_{k}$.
Une $k$-vari\'et\'e int\`egre $X$ est dite {\it r\'etractilement rationnelle}
s'il existe des ouverts de Zariski non vides $U \subset X$ et $V \subset \P^m_{k}$ ($m$ convenable),
 et des $k$-morphismes $ f : U \to V$ et $g : V \to U$ tels que le compos\'e $g \circ f$ est l'identit\'e de $U$.
 Une $k$-vari\'et\'e int\`egre stablement $k$-rationnelle sur un corps infini $k$ est r\'etractilement rationnelle.

\medskip

Pour $X$ une $k$-vari\'et\'e,
 on note $Z_{0}(X)$ le groupe des z\'ero-cycles sur $X$,
c'est-\`a-dire le groupe ab\'elien libre sur les points ferm\'es de $X$.
On note $CH_{0}(X)$ le groupe de Chow des z\'ero-cycles, quotient de $Z_{0}(X)$
par l'\'equivalence rationnelle.
Pour $X$ une $k$-vari\'et\'e propre,  l'application qui \`a un point ferm\'e $P$
associe le degr\'e $[k(P):k]$ du corps r\'esiduel $k(P)$ s'\'etend en un homomorphisme
$ \mathrm{deg}_{k} : CH_{0}(X) \to \Z$. On note $A_{0}(X)$ le noyau de cet homomorphisme.

On dit qu'une $k$-vari\'et\'e propre g\'eom\'etriquement int\`egre $X$  est   universellement $CH_{0}$-triviale si pour tout corps $F$ contenant $k$ 
la fl\`eche $ \mathrm{deg}_{F} : CH_{0}(X_{F}) \to \Z$ est un isomorphisme
(voir \cite{ACTP}). 
Sous l'hypoth\`ese que la $k$-vari\'et\'e $X$, de corps des fonctions $k(X)$, 
est lisse et  poss\`ede un z\'ero-cycle $z$ de degr\'e 1,
 ceci est \'equivalent \`a l'\'enonc\'e : $A_{0}(X_{k(X)})=0$,
 ou encore \`a l'existence d'une d\'ecomposition (enti\`ere) de la diagonale
 de $X\times_{k}X$, c'est-\`a-dire \`a l'\'enonc\'e : le point g\'en\'erique de
 $X$, vu comme point de $X_{k(X)}$, est rationnellement \'equivalent
sur $X_{k(X)}$  au z\'ero-cycle $z\times_{k}k(X)$.
  \cite[Lemma 1.3]{ACTP}.

Une $k$-vari\'et\'e g\'eom\'etriquement int\`egre propre et lisse qui est r\'etractilement rationnelle
est universellement $CH_{0}$-triviale \cite[Lemme 1.5]{CTPENS}.

Les groupes $H^{i}(k,\mu)$ sont les groupes de cohomologie galoisienne, o\`u $\mu$ est un  module galoisien de torsion premi\`ere \`a la
  caract\'eristique de $k$. On note $$H^{i}_{nr}(k(X)/k,\mu)\subset H^{i}(k(X),\mu)$$ les groupes de cohomologie non ramifi\'ee
d'une $k$-vari\'et\'e int\`egre $X$ de corps des fonctions $k(X)$ (voir \cite{Santa}).

Soit $X$ une $k$-vari\'et\'e propre, lisse, g\'eom\'etriquement connexe.
 Si $X$ est  r\'etractilement rationnelle,
alors, pour tout module galoisien  $\mu$ de torsion  premi\`ere \`a la
caract\'eristique,   tout entier $i \geq 0$, et tout corps $L$ contenant $k$,
l'application $H^{i}(L,\mu) \to H^{i}_{nr}(L(X)/L,\mu)$
est un isomorphisme. En particulier l'application $\Br(L) \to \Br_{nr}(L(X)/L)=\Br(X_{L})$
est un isomorphisme sur la torsion premi\`ere \`a la caract\'eristique.

Pour $k=\R$,  un invariant birationnel stable  des $\R$-vari\'et\'es projectives et lisses g\'eom\'etriquement connexes 
est le nombre $s$ de composantes connexes de l'espace topologique $X(\R)$
associ\'e \`a une telle $\R$-vari\'et\'e $X$.
Cet invariant $s$ peut \^etre calcul\'e de diverses mani\`eres.  Supposons $X(\R) \neq \emptyset$.

On a  $ A_{0}(X)/2 \simeq (\Z/2)^{s-1}$  \cite{CTIschebeck}.

Pour tout $i > {\rm dim}(X)$, on a $H^{i}_{nr}(\R(X)/\R, \Z/2) \simeq (\Z/2)^{s}$ \cite{CTP}.

Si $X/\R$ est stablement rationnelle sur $\R$,
 ou m\^eme simplement  r\'etractilement rationnelle sur $\R$, alors l'espace topologique $X(\R)$ est connexe non vide. 
 
 On trouvera dans \cite{ACTP} et  \cite{CTter}  des  d\'emonstrations et des r\'ef\'erences  pour les \'enonc\'es ci-dessus.

Pour tout groupe ab\'elien $A$ et tout entier $n>1$, on note $A[n] \subset A$ le sous-groupe
form\'e des \'el\'ements annul\'es par $n$.

 \subsection{Fibrations en quadriques}

 Soit $f: X \to \P^1_{\R}$ une famille de quadriques de dimension relative $r\geq 1$ sur la droite projective r\'eelle,
de fibre g\'en\'erique lisse, et d'espace total $X/\R$ projectif et lisse.
La fibration $f_{\C} : X_{\C} \to \P^1_{\C}$ admet une section (Max Noether, Tsen),
donc  la $\C$-vari\'et\'e $X_{\C}$ est   $\C$-rationnelle. 
Si $r=1$ et si l'espace $X(\R)$ est non vide et connexe, la surface $X$ est rationnelle sur $\R$ (Comessatti \cite{Com}).
 
Pour toute famille de quadriques  $X \to \P^1_{\R}$ de dimension relative
au moins 2, et aussi pour les familles de quadriques
$X \to \P^m_{\R}$ de dimension relative au moins 1 qui poss\`edent une
section rationnelle sur les complexes, on peut  se poser la question si la connexit\'e de $X(\R)$
 suffit \`a assurer la rationalit\'e de $X$. Dans \cite{BW}, O. Benoist et O. Wittenberg
ont \'etudi\'e des exemples de telles fibrations de dimension relative~1 sur $\P^2_{\R}$
poss\'edant une section sur $\C$, donc rationnelles sur $\C$ (voir aussi \cite{FJSVV} et \cite{JJ}).
La technique des torseurs de jacobiennes interm\'ediaires \cite{HT, BW} permet dans certains
cas de montrer la non $\R$-rationalit\'e de certains solides -- mais ne dit rien sur la rationalit\'e stable,
ni non plus sur la $CH_{0}$-trivialit\'e universelle.

  \bigskip 
  
  Dans cet article, on s'int\'eresse aux solides $X$
  fibr\'es en surfaces quadriques sur la droite projective $\P^1_{k}$ sur un corps $k$ de caract\'eristique diff\'erente de 2.
  Ceux-ci ont \'et\'e \'etudi\'es dans \cite{Sk, CTSk}.
  
  On dit qu'une telle fibration $X \to \P^1_{k}$ est de type (I) si toutes
  les fibres g\'eom\'e\-triques sont int\`egres, auquel cas on peut
  se ramener \`a supposer que toute 
  fibre singuli\`ere
   est un c\^{o}ne sur
  une conique lisse.  
Si $X$ est de type (I), alors  pour toute extension   $L/k$ du corps de base, on a  $\Br(X_{L})/\Br(L)=0$
(voir  \cite[Cor. 3.2 (a), cas $f=0$]{Sk} et \cite[Thm. 2.3.1, cas $T=\emptyset$]{CTSD}). 
Si $X$ n'est pas de type (I),  la non $k$-rationalit\'e stable de $X$ peut souvent \^etre d\'etect\'ee par un calcul de groupe de Brauer \cite[Cor. 3.1, Lemma 3.5]{Sk} \cite[Thm. 2.3.1]{CTSD}.
 {\it Dans le cas $k=\R$ nous donnons ainsi un exemple de telle fibration en surfaces quadriques sur $\P^1_{\R}$ avec $X(\R)$ connexe et $\Br(X)/\Br(\R)\neq 0$, et donc $X$ non stablement rationnelle }(voir la section  \ref{contreexemple}).

Comme nous  l'a  indiqu\'e O. Wittenberg (d\'ecembre 2023), pour $X \to \P^1_{\R}$ de type (I),
la technique de \cite{HT} et \cite{BW} permet de d\'etecter la non rationalit\'e sur $\R$ de nombreuses familles de vari\'et\'es fibr\'ees  en surfaces  quadriques sur $\P^1_{\R}$ pour lesquelles de plus $X(\R)$ est connexe
non vide. 
On peut  en fait d\'ej\`a
donner de tels exemples de non rationalit\'e \`a partir d'intersections lisses de deux quadriques dans $\P^5_{\R}$, comme on voit en combinant le th\'eor\`eme 36  et 
 la section
 11.4  de \cite{HT}. 

Ces divers exemples donnent une r\'eponse n\'egative \`a une question de S. Zimmermann
sur la rationalit\'e des solides r\'eels $X$  fibr\'es en quadriques sur la droite projective
 lorsque l'espace topologique
associ\'e $X(\R)$ est connexe.

\subsection{Une classe particuli\`ere de fibrations en surfaces quadriques}\label{classebeta}
   
  Soit $k$ un corps de caract\'eristique diff\'erente de $2$. Supposons   $X \to \P^1_{k}$   de type (I) et  supposons qu'il y a au moins une fibre g\'eom\'etrique non lisse. On associe \`a une telle fibration une courbe discriminant $\Delta$, lisse et g\'eom\'etriquement connexe,  qui est un  rev\^{e}tement  double de $\P^1_{k}$, et une  classe $\beta \in \Br(\Delta)$, qui est une classe de quaternions dans $\Br(k(\Delta))$.

  Lorsque la classe $\beta \in \Br(\Delta)$ n'est pas dans l'image de $\Br(k)$
   et que le nombre de fibres g\'eom\'etriques singuli\`eres de $X \to \P^1$
   est au moins 6, la technique de \cite{BW}
    permet     souvent  de
   d\'etecter la non $k$-rationalit\'e de $X$
    (Wittenberg, voir le th\'eor\`eme \ref{wittenberg23} ci-dessous).
   
   Le cas o\`u la classe $\beta$ est constante, de la forme $\beta=(a,b)$
   avec $a,b\in k^{\times}$,
   correspond exactement
    aux fibrations qu'on peut donner en coordonn\'ees
   affines par une \'equation
   $$x^2-ay^2-bz^2 = q(u)$$
   avec  $a,b \in k^{\times}$ et $q(u)$ polyn\^ome s\'eparable.

  Dans le cas $k=\R$, $a=b=-1$ et $X(\R) \neq \emptyset$,   il y a un nombre pair $2d$ de
  points de 
  $\P^1({\R})$ 
  dont la fibre n'est pas lisse. Ce sont les points o\`u $q(u)$ s'annule, et le point
  \`a l'infini si $q(u)$ est de degr\'e impair.
  Si $d=0$, alors $q(u)$ est positif, donc est une somme de deux carr\'es dans $\R[u]$.
  La fibration $X \to \P^1_{\R}$ d\'efinie par $u$ admet  alors une section, la fibre g\'en\'erique
  est une quadrique lisse avec un point rationnel, son corps des fonctions est une extension
  transcendante pure de $\R(u)$, et donc   $X$
 est  une vari\'et\'e $\R$-rationnelle. 
 Supposons  $d>0$.
  L'espace topologique $X(\R)$ a $d$ composantes connexes.
  Si $d>1$, alors $X$ n'est pas stablement rationnelle, ni m\^{e}me
  universellement $CH_{0}$-triviale. 
  On est ainsi ramen\'e \`a consid\'erer les \'equations de la forme
  \begin{equation}\label{eq}
   x^2+y^2+ z^2 = u.p(u)
   \end{equation}
   avec $p(u)$ polyn\^{o}me s\'eparable unitaire de degr\'e pair qui est strictement
   positif sur $\R$.

   Dans ce cas, $X(\R)$ est connexe, et la question de la rationalit\'e de 
   ces fibrations en quadriques ne peut \^etre trait\'ee par les m\'ethodes
   de \cite{HT}, \cite{BW}, comme indiqu\'e 
   \`a la  section  
  \ref{IJTWitt}.

  \subsection{\'Enonc\'es et structure de l'article} 
    Dans cet article, nous \'etudions la question de la $CH_{0}$-trivialit\'e (universelle) des vari\'et\'es (\ref{eq}).  Soit $W=X \times_{\R} \Delta$.
En utilisant le travail \cite{CTSk}, nous
 r\'eduisons la question de la $CH_{0}$-trivialit\'e universelle de $X$ 
\`a la nullit\'e d'une classe pr\'ecise dans $H^3_{nr}(\R(W)/\R, \Z/2)$,
groupe dont nous \'etablissons la finitude (th\'eor\`eme \ref{finitude} et remarque \ref{remfinitude}).
On montre (voir th\'eor\`emes \ref{criterenullitecycle} et  \ref{equivalencesR}):

\begin{thm}\label{annoncethm}
Soit 
 $p(u)\in \R[u]$ s\'eparable unitaire, non constant, de degr\'e pair, avec $p(0)\neq 0$,
 positif sur $\R$.
Soit   $\pi: X \to \P^1_{\R}$ une fibration en surfaces quadriques donn\'ee
au-dessus de $\A^1_{\R}=\Spec(\R[u])$ par l'\'equation
$$ x^2+y^2+z^2 - u.p(u). t^2=0.$$ 
 Soit $\Delta$ la courbe  r\'eelle projective et lisse d'\'equation
 $$w^2= vp(-v).$$  
Soit $W= X \times_{\R} \Delta$.

Alors la classe $(u+v,-1, -1)\in H^3(\R(W),\Z/2)$ est non ramifi\'ee:
$$(u+v,-1,-1) \in H^3_{nr}(\R(W)/\R,\Z/2).$$
Les conditions suivantes sont \'equivalentes :
\begin{itemize}
\item [(i)] Cette classe est nulle.

\item [(ii)] La fonction $u+v$ est une somme de 4 carr\'es dans $\R(W)$.

\item [(iii)] Pour tout corps $F$ contenant $\R$, on a $$A_{0}(X_{F})=0.$$
\end{itemize}
 \end{thm}

 Nous donnons ensuite {\it 
 plusieurs
 classes de vari\'et\'es $X$ pour lesquelles
nous pouvons montrer que l'\'el\'ement $(u+v,-1,-1)$ est nul,
et donc \'etablir que ces vari\'et\'es $X$ sont universellement $CH_{0}$-triviales.}

Plus pr\'ecis\'ement,
 on \'etablit que {\it $X$ est   universellement $CH_{0}$-triviale} dans les cas suivants:
\begin{enumerate}
\item (th\'eor\`eme \ref{ex2}) $p(u)=u^2+au+b\in \R[u]$ est un polyn\^ome s\'eparable tel que
	$$b\geq \frac{a^2}{3};$$ 
	\item  (th\'eor\`eme \ref{ex2n}) $p(u)=u^{2n}+\sum_{i=0}^{n-1}a_{2i}u^{2i}
	\in \R[u]$ s\'eparable, satisfaisant $$a_0>0\mbox{ et } 
	a_{2i}\geq 0 \mbox{ pour tout } 0<i<n;$$
	\item (th\'eor\`eme \ref{emc}) 
	 le polyn\^{o}me $p(u)$ est de degr\'e 2,
	  la courbe elliptique
   $$ E: z^2=u.p(u)$$
   est \`a multiplication complexe, et  l'on a $$End_{\C}(E) =
\Z[\omega],$$  o\`u $\omega \in \C$ est un entier quadratique imaginaire, satisfaisant une \'equation
$$\omega^2 - d \omega +c =0, \mbox{ o\`u }d,c \in \Z\mbox{ et }d \mbox{ est impair}.$$
\end{enumerate}

Comme on verra 
\`a la section 
\ref{scomp}, pour un polyn\^ome $p$ positif,  l'hypoth\`ese faite dans le th\'eor\`eme   \ref{ex2}
\'equivaut \`a l'hypoth\`ese que l'invariant $j(E) \in \R$ de la courbe elliptique $E$ d'\'equation  $ z^2=u.p(u)$ est un r\'eel positif ou nul.

\bigskip

\`A la section 
\ref{rap}, on donne quelques rappels sur les  formes quadratiques et sur l'article \cite{CTSk}. \`A la section \ref{surk} on \'etudie les fibrations en surfaces quadriques $X\to \P^1_k$  de type (I) d'\'equation affine
$$x^2-ay^2-bz^2=u.p(u)$$
 sur un corps $k$ arbitraire de caract\'eristique diff\'erente de $2$. Sous la  condition technique que le polyn\^ome $p$ est repr\'esent\'e par la forme
quadratique
  $$<<a,b>>=<1,-a,-b,ab>$$ 
  sur le corps $k(u)$, on donne un crit\`ere pour la $CH_0$-trivialit\'e 
 universelle de la vari\'et\'e $X$.

Aux sections \ref{sectionreel} \`a \ref{scomp} on s'int\'eresse au cas $k=\mathbb R$.

 \`A la section \ref{sectionreel}, on sp\'ecialise le crit\`ere ci-dessus
 au cas
des vari\'et\'es $X/\R$ d'\'equation affine
 $$x^2+y^2+z^2=u.p(u)$$
 avec $p(u)$ de degr\'e pair,  strictement positif sur $\R$.
 
    \`A la section  \ref{sectioncarres}, pour
   deux classes 
   de telles  vari\'et\'es $X/\R$,   
  pour $W=X \times_{\R} \Delta$ comme dans le th\'eor\`eme \ref{annoncethm},
  on \'etablit  par une m\'ethode \'el\'ementaire que la fonction $u+v$ est une somme de 4 carr\'es dans $\R(W)$,
  et donc que $X$ est universellement $CH_{0}$-triviale.
 
 La section \ref{lemmescoho} regroupe des \'enonc\'es de cohomologie galoisienne.
 
\`{A} la section \ref{fibconiques}, 
on utilise un mod\`ele lisse birationnel (non propre) de  $X$ qui est une fibration en coniques sur le plan $\A^2_{\R}$,
 laquelle  d\'eg\'en\`ere le long d'un ouvert d'une courbe hyperelliptique $\Gamma$ qui est une tordue de la courbe $\Delta$.
 Pour $F$ variant parmi les surcorps quelconques de $\R$,  on \'etudie
 le groupe $H^3_{nr}(F(X)/F, \Q/\Z(2))$  en utilisant  le complexe de Bloch-Ogus sur $\A^2_{F}$.
 Ceci m\`ene \`a l'\'etude d'un sous-groupe du groupe de Brauer de $\Gamma_{F}$
 qui est d\'etermin\'e par la jacobienne de $\Gamma_{F}$.
 Le th\'eor\`eme principal est le th\'eor\`eme \ref{BOH4}.
 Nous en d\'eduisons (th\'eor\`eme \ref{finitude}) la finitude du 
    groupe $H^3_{nr}(\R(X \times_{\R}Y)/\R, \Q/\Z(2))$ 
lorsque $H^3_{nr}(Y,\Q/\Z(2))$ est fini, par exemple pour $Y$
une courbe.

\`A la section \ref{sansMC} on \'etudie plus pr\'ecis\'ement le cas o\`u $Y=\Delta$ et o\`u
la jacobienne  de $\Delta$ n'a pas de multiplication complexe.

Lorsque $\Delta$ est une courbe elliptique \`a multiplication complexe,
 on donne \`a la section \ref{avecMCimpaire}
 une condition suffisante pour que $X$ soit universellement $CH_{0}$-triviale.
Une description explicite des cas o\`u cette condition est r\'ealis\'ee nous a \'et\'e communiqu\'ee
 par Yuri Zarhin.  C'est l'objet de la section \ref{zarhin}.

\`{A} la section  \ref{scomp}, pour $p(u)$ de degr\'e 2,  on compare les r\'esultats de $CH_{0}$-trivialit\'e des sections  \ref{sectioncarres} et
\ref{avecMCimpaire}. La comparaison fait intervenir
 l'invariant $j$ de la courbe elliptique r\'eelle $z^2=u.p(u)$.

 \`A la section   \ref{sectioncomp}, on rappelle quelques r\'esultats de rationalit\'e pour des vari\'et\'es fibr\'ees en quadriques sur $\P^1_{\R}$.

   \`A la section \ref{deuxnonrat} nous donnons notre exemple de vari\'et\'e 
   fibr\'ee
   en surfaces quadriques sur $\P^1_{\R}$
 qui n'est pas  de type (I) et
dont le groupe de Brauer est non trivial, et qui donc n'est pas universellement $CH_{0}$-triviale, et donc n'est pas
r\'etractilement rationnelle.
Cet exemple admet des mod\`eles birationnels singuliers $Y$ du type suivant : hypersurface cubique dans $\P^4_{\R}$
et intersection de deux quadriques dans $\P^5_{\R}$, dont tous les points 
ferm\'es
singuliers sont des points complexes,
 et qui satisfont
que $Y(\R)$ est connexe. 
Nous terminons cette section en citant un r\'esultat r\'ecent de non rationalit\'e de Wittenberg pour des fibrations de type (I), r\'esultat obtenu par la m\'ethode
des torseurs de jacobienne interm\'ediaire.

\`A la section \ref{compdeuxquad}, nous donnons des compl\'ements \`a \cite{CT23}, qui impliquent que les fibrations 
(\ref{eq})
consid\'er\'ees dans cet article ne sont pas birationnelles \`a celles obtenues \`a partir d'une intersection lisse de deux
quadriques dans $\P^5$. 

La section \ref{corpsdenombres} contient le calcul   des groupes de Chow des z\'ero-cycles pour les fibrations de type (I)
d\'efinies sur un corps de nombres.

\bigskip

  La question de savoir s'il existe une fibration en surfaces quadriques $X\to \mathbb P^1_{\R}$ de type (I) avec $X(\R)$ connexe et qui n'est pas universellement $CH_0$-triviale  reste ouverte.

  La question de savoir si les vari\'et\'es du th\'eor\`eme \ref{annoncethm} sont toutes
 universellement $CH_{0}$-triviales est ouverte. 
    
 La question de savoir si les vari\'et\'es du th\'eor\`eme \ref{annoncethm} qui sont universellement $CH_{0}$-triviales
 sont rationnelles, stablement rationnelles, r\'etractilement rationnelles, reste \'egalement ouverte, d\'ej\`a dans le cas donn\'e par l'\'equation affine
  $$ x^2+y^2+z^2=u(u^2+1).$$

\subsection*{Remerciements} 
Nous remercions Olivier Wittenberg pour plusieurs \'echanges \'electroniques, 
et pour nous avoir communiqu\'e la proposition \ref{510fin}.
 Nous sommes reconnaissants \`a Yuri Zarhin pour
 le contenu de la section \ref{zarhin}.

 Jean-Louis Colliot-Th\'el\`ene remercie la Fondation Simons pour son hospitalit\'e \`a New York en mars et avril 2024.  Alena Pirutka a b\'en\'efici\'e du soutien de  NSF grant DMS-2201195.

	  \section{Rappels}\label{rap}

\subsection{Formes quadratiques et $H^3$.}
   
  Soit $F$ un corps de caract\'eristique diff\'erente de $2$. Soit $a \in F^{\times}$.
  On note $<<a>>$ la forme quadratique binaire  $x^2-ay^2$ sur le corps $F$.
Pour $a_{1}, \ldots, a_{n}$ dans $F^{\times}$,
 on a   la forme de Pfister
   en $2^n$ variables :
$$ <<a_{1}, \dots, a_{n}>>:= <<a_{1}>> \otimes \dots \otimes <<a_{n}>>.$$

Ainsi $<<a,b>>= x^2-ay^2-bz^2+abt^2$ est la forme norme r\'eduite associ\'ee \`a l'alg\`ebre de quaternions
$(a,b)_{F}$.

\begin{prop}\label{classique}
Soit $F$ un corps de caract\'eristique diff\'erente de $2$.  Pour des \'el\'ements $a,b,c\in F^{\times}$, les propri\'et\'es
 suivantes sont \'equivalentes :
 \begin{itemize}
 \item [(i)] La forme quadratique $<<a,b,c>>$ est isotrope.
 \item [(ii)] La forme  quadratique $<<a,b,c>>$ est hyperbolique.  
 \item [(iii)] L'\'el\'ement $c \in F^{\times}$ est repr\'esent\'e par la forme
 $<<a,b>>$.
\item [(iv)] Le cup-produit  $(a) \cup (b) \cup(c)   \in H^3(F,\Z/2)$ est nul.
\end{itemize}
\end{prop}

\begin{proof}
L'\'equivalence des \'enonc\'es $(i)$, $(ii)$, $(iii)$ est un cas particulier
de r\'esultats de Pfister. Que ces \'enonc\'es impliquent $(iv)$
est connu \cite{Ar}.
Que $(iv)$ implique les autres \'enonc\'es est un r\'esultat
de Merkurjev et Suslin (voir \cite[Th\'eor\`eme 12.2]{MS}).
\end{proof}

On note $\Q/\Z(2)$ la limite inductive des modules galoisiens $\mu_{2^n}^{\otimes 2}$.
La multiplication par $2$ induit une suite exacte
$$ 0 \to \Z/2 \to \Q/\Z(2) \to \Q/\Z(2) \to 0,$$
et donc une application surjective
 $H^3(F,\Z/2) \to H^3(F,\Q/\Z(2))[2]$.

 Le r\'esultat suivant, d'abord  d\'emontr\'e par Merkurjev (voir \cite{MS} p.1045), est   un cas particulier
de la conjecture de Bloch-Kato sur les corps (connue maintenant en toute
g\'en\'eralit\'e).

\begin{prop}\label{lemmaMerk}
Soit $F$ un corps de caract\'eristique diff\'erente de $2$.
L'application surjective
$$H^3(F,\Z/2) \to H^3(F,\Q/\Z(2))[2]$$
est  un isomorphisme.
\end{prop}

Pour $q$ forme de rang 3 et $q$ forme de rang 4,
on utilisera le r\'esultat suivant de Kahn, Rost, et Sujatha:

\begin{prop}\label{KRSu} Soit $q$ une forme quadratique non d\'eg\'en\'er\'ee 
de rang au moins $3$ sur un corps $F$ de caract\'eristique diff\'erente de $2$.  Soit $Q$ la quadrique lisse associ\'ee.
Si le rang de $q$ est $6$, supposons que $q$
n'est pas semblable \`a une forme d'Albert.
L'application  $$H^3(F,\Q/\Z(2)) \to H^3_{nr}(F(Q)/F,\Q/\Z(2))$$ est surjective.
\end{prop}

\begin{proof}
Voir \cite[Thm. 5 p. 846; Cor. 10.2 et Prop. 10.3)]{KRS}.\bigskip
\end{proof}

\subsection{Fibr\'es en surfaces quadriques sur la droite projective}\label{rappelsCTSk}

On donne ici quelques rappels du travail  \cite{CTSk}.

\medskip

Soit $k$ un corps de caract\'eristique diff\'erente de 2. 
Soit $k_{s}$ une cl\^{o}ture s\'eparable. 
Soit $$\pi :  X \to \P^1_{k}$$ une famille de surfaces quadriques au-dessus de $\P^1_{k}$.
On suppose que $X$ est une $k$-vari\'et\'e projective et lisse g\'eom\'etriquement int\`egre. 
On  trouvera  des mod\`eles explicites concrets dans \cite{Sk} et  \cite{CTSk}.

Sur le corps $k_{s}$, une telle fibration admet une section, donc $X^s:=X\times_{k}k_{s}$ est 
$k_{s}$-rationnelle, i.e. $k_{s}$-birationnelle \`a un espace projectif.
On s'int\'eresse \`a la question si une telle $X$ est $k$-rationnelle.

 On se limite ici au cas des fibrations de type (I),
celles dont toutes les fibres g\'eom\'etriques sont int\`egres, ce qui signifie que les fibres g\'eom\'etriques singuli\`eres
sont des c\^{o}nes sur des coniques lisses.

On suppose qu'il y a au moins une fibre g\'eom\'etrique singuli\`ere.

Pour les fibrations de type (I),   le module galoisien $\Pic(X\times_{k}k_s)$
est $\Z \oplus \Z$, avec action triviale du groupe de Galois.
Pour tout corps $F$ contenant $k$, l'application  $\Br(F) \to \Br(X_{F})$  induite par le morphisme $X_{F} \to \Spec(F)$
est surjective.

La quadrique g\'en\'erique a un discriminant bien d\'efini dans $k(\P^1)^{\times}/k(\P^1)^{\times 2}$.
Comme la fibration est de type (I) 
et a au moins une fibre g\'eom\'etrique singuli\`ere,
ce discriminant n'est pas trivial dans $k_{s}(\P^1)^{\times}/k_{s}(\P^1)^{\times 2}$. Il d\'efinit une extension quadratique $k(\Delta)/k(\P^1)$
et un rev\^{e}tement double 
\begin{equation}\label{delta}
\Delta \to \P^1_{k}
\end{equation} avec $\Delta/k$ une courbe  lisse g\'eom\'etriquement
int\`egre.

On associe \`a cette situation une fibration  en coniques $q: Y \to \Delta$.
La fibre g\'en\'erique de $\pi : X \to \P^1_{k}$ est la descendue \`a la Weil
de $k(\Delta)$ \`a $k(\P^1)$  de la fibre g\'en\'erique de $q$
\cite[Thm. 2.5]{CTSk}.
On a une classe associ\'ee  $\alpha \in \Br(k(\Delta))$
qui est la classe d'une alg\`ebre de quaternions $D/k(\Delta)$.
Pour $\pi$ de type (I), la fibration en coniques est lisse et on a
$$\alpha\in \Br(\Delta)[2] \subset \Br(k(\Delta))[2] = H^2(k(\Delta),\Z/2).$$

On d\'efinit le sous-groupe $k(\P^1)^{\times}_{dn} \subset k(\P^1)^{\times}$ 
form\'e des fonctions rationnelles sur $\P^1$  dont le diviseur est 
dans l'image
de $\pi_{*} : Z_{0}(X) \to Z_{0}(\P^1)$. 
Il y a une application naturelle surjective
\begin{equation}\label{p1dn}
k(\P^1)^{\times}_{dn} \to \Ker[\pi_{*} : CH_{0}(X) \to CH_{0}(\P^1)=\Z],
\end{equation}
dont on d\'ecrit  le noyau
\cite[Thm. 4.2]{CTSk}.
 
On d\'efinit le groupe analogue $k(\Delta)^{\times}_{dn}$ pour $q : Y \to \Delta$.
 Le groupe $k(\Delta)^{\times}_{dn}$ contient le sous-groupe $N_{D}(k(\Delta)) \subset k(\Delta)^{\times}$
 des normes r\'eduites. On a une application naturelle surjective
 $$k(\Delta)^{\times}_{dn} \to \Ker[\pi_{*} : CH_{0}(Y) \to CH_{0}(\Delta)],$$
 dont le noyau est $k^{\times}. N_{D}(k(\Delta))$.
 
 L'application $k(\P^1)^{\times} \to k(\Delta)^{\times}$ induit une
 application 
\begin{equation}\label{ddn} 
 k(\P^1)^{\times}_{dn} \to k(\Delta)^{\times}_{dn}.
 \end{equation}
 
 Par ailleurs, on a une application 
 \begin{equation}\label{modNdansH3}
k(\Delta)^{\times}/N_{D}(k(\Delta)) \to H^3(k(\Delta),\Z/2),\; g\mapsto (g) \cup \alpha.
\end{equation}
D'apr\`es  la proposition \ref{classique} cette
application est injective.
 
 \begin{thm}\label{TheoremeCTSk}  \cite[Thm. 4.2, Thm. 4.3, p.493]{CTSk}.
 Soit $\pi: X \to \P^1_{k}$ une fibration en surfaces quadriques de type (I).
 \begin{itemize}
 \item[(i)] On a un plongement 
 $$\Phi :  A_{0}(X)  \hookrightarrow k(\Delta)^{\times}_{dn}/k^{\times}. N_{D}(k(\Delta))$$
 induit par les applications (\ref{p1dn}) et (\ref{ddn}).
 \item [(ii)] La compos\'ee de l'application $\Phi$ et  de l'application (\ref{modNdansH3}) induit un plongement 
 $$\Psi :  A_{0}(X)\hookrightarrow H^3_{nr}(k(\Delta)/k,\Z/2)/ \Im[H^1(k,\Z/2)\cup \alpha)]$$
 de $A_0(X)$ dans le quotient de $H^3_{nr}(k(\Delta),\Z/2)$
 par le sous-groupe des \'el\'ements de $H^3(k(\Delta),\Z/2)$
  de la forme $(e) \cup \alpha$  avec  $e \in k^{\times}/k^{\times 2}$. 
 \end{itemize}
 \end{thm}

\section{Une famille sp\'ecifique de fibrations de type (I) }\label{surk}

Soit $k$ un corps de caract\'eristique diff\'erente de 2. 
Dans la suite de cet article, pour $a,b \in k^{\times}$,
on s'int\'eresse aux fibrations en surfaces quadriques 
de type (I), d'\'equation affine
\begin{equation}\label{eqsurk}
x^2-ay^2-bz^2 - u.p(u)=0
\end{equation}
avec $p(u)$ s\'eparable unitaire, non constant,  de degr\'e pair $2d$,
tel que $p(0)\neq 0$.
Soit $$q(v)=v^{2d}.p(1/v) \in k[v].$$
Sauf mention du contraire, on prendra pour $\pi : X \to \P^1_{k}$
le mod\`ele projectif et lisse  standard $X/k$ donn\'e par recollement
 de la vari\'et\'e  d\'efinie par l'\'equation
$$ x^2-ay^2-bz^2 - u.p(u) t^2 =0$$
dans $\P^3_{k} \times_{k} \Spec(k[u])$ avec la vari\'et\'e d\'efinie 
par
l'\'equation
$$ {x'}^2-a{y'}^2-b{z'}^2 - v.q(v) {t'}^2 =0$$
dans $\P^3_{k} \times_{k} \Spec(k[v])$, via 
$v=1/u$ et
$$(x',y',z',t';v) = (x.v^{d+1}, y.v^{d+1}, z.v^{d+1}, t;1/u).$$
Le morphisme
$\pi : X \to \P^1_{k}$, est donn\'e par $u=1/v$. Toutes les fibres
sont g\'eom\'e\-tri\-quement int\`egres.

 Si $(a,b)=0 \in H^2(k,\Z/2)$, alors la fibration  $\pi: X \to \P^1_{k}$ admet une section. La $k$-vari\'et\'e $X$
 est alors $k$-rationnelle. On s'int\'eresse au cas $(a,b)\neq 0 \in \Br(k)$.

 \subsection{Cons\'equences de \cite{CTSk}}
Soit $\pi: X \to \P^1_{k}$
de type (I) donn\'ee par l'\'equation affine (\ref{eqsurk}).

La fibre \`a l'infini est le c\^{o}ne sur la conique d'\'equation $x^2-ay^2-bz^2=0$. Le sommet du c\^{o}ne est un $k$-point de la $k$-vari\'et\'e lisse $X$. Notons-le $m_{\infty}$.
Notons $c_{\infty}$ le point \`a l'infini de $\P^1_{k}$:

$$\xymatrix{
X \ni m_{\infty}\ar@<2ex>[d]\ar@<-3.5ex>[d]^{\pi}\\
\P^1_{k}\ni c_{\infty}
}
$$

Dans ce cas, le discriminant est $-a.b.u.p(u)$,
la courbe $\Delta$ dans (\ref{delta}) est la courbe hyperelliptique d'\'equation affine
$$w^2= -ab.u.p(u),$$
et la classe $\alpha\in \Br(\Delta)$ est l'image de la classe de quaternions     
 $$D=(a,b) \in  H^2(k,\Z/2) \subset \Br(k).$$
 
 Pour  $q : Y \to \Delta  $ on peut prendre la  projection $ V_{a,b} \times_{k} \Delta \to \Delta$, 
 o\`u l'on note $V_{a,b}$ la conique d'\'equation $x^2-ay^2-bz^2=0$.

 D'apr\`es le th\'eor\`eme \ref{TheoremeCTSk} on a ici un plongement
 \begin{multline}
  \Psi :  A_{0}(X) \stackrel{\Phi}{\hookrightarrow}  k(\Delta)^{\times}_{dn}/k^{\times}. N_{D}(k(\Delta)) \hookrightarrow \\ \hookrightarrow H^3_{nr}(k(\Delta),\Z/2)/ \Im[H^1(k,\Z/2)\cup (a) \cup (b)].
 \end{multline}

 Soit $m \in X(k)$ d'image $\pi (m) = c  \in \A^1(k) = k $ avec $c.p(c) \neq 0$.
 La diff\'erence $m-m_{\infty}$ est un z\'ero-cycle de degr\'e z\'ero sur $X$.

 \begin{lemma}\label{mainlemma}
 Soit $m \in X(k)$ d'image $\pi (m) = c  \in \A^1(k) = k $ avec $c.p(c) \neq 0$. Les conditions suivantes sont \'equivalentes :
 \begin{itemize}
  \item[(i)]  L'image  de la classe de $m-m_{\infty}$ dans $A_{0}(X)$  par l'application $\Psi$
 est nulle.
 \item[(ii)] 
 La classe $$(c.(c-u), a,b)=   (p(c).(c-u), a,b)  \in H^3_{nr}(k(\Delta)/k,\Z/2)$$ est nulle.
 \end{itemize}
Si l'on suppose de plus que l'on a $(p(u), a, b)=0 \in H^3(k(u),\Z/2)$,
 alors ces conditions sont \'equivalentes \`a  $$(c-u,a,b)=0 \in H^3_{nr}(k(\Delta)/k,\Z/2).$$
  \end{lemma}
  \begin{proof} 
 Le diviseur de la fonction $u-c \in k(\P^1)^{\times}_{dn}$ sur $\P^1_{k}$ est $c -c_{\infty} $.
 La description de l'application $\Psi$ dans la section \ref{rappelsCTSk}  montre que   la classe de $m-m_{\infty}$ dans $A_{0}(X) $ est nulle 
 si et seulement si l'on a :
 
 (H) Il existe une constante $e \in k^{\times}$ telle que
 pour la fonction $u-c \in k(\P^1)^{\times} \subset k(\Delta)^{\times}$ on ait
 \begin{equation}\label{equmc}
 (u-c, a, b) = (e,a,b) \in H^3_{nr}(k(\Delta), \Z/2)
 \end{equation}  pour $e \in k^{\times}$
 convenable.  
 
 Supposons (H) satisfaite.
 Comme l'\'equation affine de la courbe lisse $\Delta$ est $$w^2=-abu.p(u),$$ la courbe  $\Delta$
 poss\`ede le  point $k$-rationnel $w=u=0$.   Comme on a $c\neq 0$,  la fonction $u-c$ est
 inversible en ce point $w=u=0$.     L'\'egalit\'e (\ref{equmc}) et  un argument de sp\'ecialisation 
  impliquent donc que l'on a $$ (-c, a,b)=(e,a,b)  \in H^3(k,\Z/2).$$
 On a donc   $$ (u-c, a, b)= (-c,a,b) \in H^3_{nr}(k(\Delta)/k,\Z/2),$$
 ce qui est \'equivaut \`a $(c(c-u), a, b)= 0$.
 Comme $c.p(c)$ est repr\'esent\'e par la forme quadratique $<1,-a,-b>$, donc est
 une norme r\'eduite de l'alg\`ebre de quaternions $(a,b)$,
 on a $(c,a,b)=(p(c),a,b)$. On obtient ainsi 
 $$(c.(c-u), a,b)=   (p(c).(c-u), a,b)  = 0 \in H^3_{nr}(k(\Delta)/k,\Z/2).$$
 
 \medskip
 
 Inversement, si l'on a $(c.(c-u),a,b)=0 \in H^3_{nr}(k(\Delta)/k,\Z/2)$, 
 alors 
 $$(u-c,a,b)=(-c,a,b) \in H^3_{nr}(k(\Delta)/k,\Z/2),$$
  et donc
 (H) est satisfaite avec $e=-c$.

 \medskip
 
Si l'on a  de plus $(p(u), a, b)=0 \in H^3(k(u),\Z/2)$ alors on a $(p(c),a,b)=0$
par \'evaluation au point $u=c$. Ainsi on a  $(c,a,b)=0$  et la condition 
devient $(c-u,a,b)=0 \in H^3(k(\Delta),\Z/2)$.
 \end{proof}

 \medskip

\subsection{Un invariant dans $H^3_{nr}(k(X\times_k \Delta)/k, \Z/2)$}
{\it Pour la suite il est n\'ecessaire de changer les variables et poser $u=-v$ dans la d\'efinition de la courbe discriminant $\Delta$.  D'apr\`es le lemme \ref{mainlemma} on s'int\'eresse aux \'el\'ements de $H^3(k(\Delta),\Z/2)$  de la forme $(c+v,a,b)$.}

 \medskip

\begin{situation}\label{Situation 1}
Soient $a,b \in k^{\times}$.   Soit 
 $p(u)\in k[u] $ s\'eparable unitaire, non constant, de degr\'e pair, avec $p(0)\neq 0$,
 satisfaisant $$(p(u),a,b)=0 \in H^3(k(u),\Z/2).$$
 Soit   $\pi: X \to \P^1_{k}$ une fibration en surfaces quadriques donn\'ee
au-dessus de $\A^1_{k}=\Spec(k[u])$ par l'\'equation
\begin{equation}\label{E1}
x^2-ay^2-bz^2 - u.p(u) t^2=0 
\end{equation}
Soit $\Delta$ la courbe  projective et lisse d'\'equation affine
\begin{equation}\label{E2}
w^2=abvp(-v) 
\end{equation}
 Soit $W= X \times_{k} \Delta$.
\end{situation}

La question de savoir si la classe   $(u+v,a,b) \in H^3(k(W),\Z/2)$  est nulle va \^etre centrale
(Th\'eor\`eme  \ref{maincritere}).  L'\'enonc\'e suivant donne une restriction assez forte sur cette classe.
\begin{thm}\label{criterenullitecycle}
  Dans la situation \ref{Situation 1},
la classe $(u+v,a,b) \in H^3(k(W),\Z/2)$ est non ramifi\'ee par rapport \`a $k$,
elle appartient \`a $H^3_{nr}(k(W)/k, \Z/2)$.
  \end{thm}

\begin{proof}
Soit $L=k(W)$ le corps des fonctions de $W$.
On a les inclusions $k(X) \subset k(W)$ et $k(\Delta) \subset k(W)$
provenant des deux projections. Dans le symbole
$(u+v,a,b)$, l'\'el\'ement $u$ provient de $k(\P^1) \subset k(X) \subset k(W)$
et l'\'el\'ement $v$ de $k(\P^1) \subset k(\Delta) \subset k(W)$.

Soit $A \subset L$ un anneau de valuation discr\`ete contenant $k$
et de corps des fractions $L$. Soit $\omega$ la valuation $L^{\times} \to \Z$
associ\'ee. Soit $\kappa$ le corps r\'esiduel. Notons
$\partial : H^3(L, \Z/2) \to H^2(\kappa,\Z/2)$ l'application r\'esidu.
On veut montrer :
$$ \partial(u+v,a,b)=0.$$
On a :
$$\partial(u+v, a,b) = \omega(u+v) (a,b)_{\kappa}.$$
Si $(a,b)_{\kappa} =0$,
ce qui \'equivaut \`a dire que la forme quadratique
$<1,-a,-b,ab>$ repr\'esente z\'ero sur $\kappa$,
le r\'esidu est clairement nul.

On suppose d\'esormais :
$$ (a,b)_{\kappa} \neq 0. $$

On est donc ramen\'e
\`a montrer que, sous cette condition,
 $\omega(u+v)$ est pair.

\medskip

Supposons le contraire:
\begin{equation}\label{e}
\omega(u+v)\mbox{ est impair.}
\end{equation}

\medskip

On voit que comme  $<1,-a,-b,ab>$ ne repr\'esente  pas z\'ero sur $\kappa$,
   tout \'el\'ement de $L^{\times}$ de la forme $P^2-a Q^2-bR^2+abS^2$ avec $P,Q,R,S \in L$
 a une valuation paire.

On \'ecrit 
$$p(u)= u^{2n} + a_{2n-1} u^{2n-1} + \dots + a_{0},$$
avec $a_{0} \neq 0$.
L'hypoth\`ese $(p(u),a,b)=0$ implique que l'on peut \'ecrire
\begin{equation}\label{ee}
p(u) = P(u)^2-a Q(u)^2 - bR(u)^2 + ab S(u)^2 \in k(u)
\end{equation}
avec $P(u), Q(u), R(u), S(u) \in k(u) $ 
des fractions rationnelles.

On sait (th\'eor\`eme de Cassels-Pfister, \cite[Chap. IX, Thm. 1.3]{Lam}) que l'on peut alors trouver une
telle repr\'esentation avec $P(u), Q(u), R(u), S(u) \in k[u]$.
 Comme $p(u)$ est s\'eparable,
aucun polyn\^ome  de $k[u]$  non constant ne divise tous les
$P(u), Q(u), R(u), S(u)$. Dans $k[u]$ on a donc une \'egalit\'e de B\'ezout :
\begin{equation}\label{eee}
\alpha(u) P(u) + \beta(u) Q(u) + \gamma(u) R(u) + \delta (u) S(u)=1,
\end{equation}
avec $\alpha(u), \beta(u), \gamma(u), \delta (u) \in k[u].$

En particulier, de l'\'equation (\ref{ee})  appliqu\'ee \`a $u$ et \`a $u=-v$ on d\'eduit que
\begin{equation}\label{ppair}
 \omega(p(u))\mbox{ est  pair et }\omega(p(-v))\mbox{ est pair}.
 \end{equation}

Par ailleurs, on obtient aussi que  $a_{0}=p(0) \in k^{\times}$ est
repr\'esent\'e par la forme quadratique $<1,-a,-b,ab>$. 

Comme $p(u)$ est repr\'esent\'e par $<1,-a,-b,ab> $ 
sur $k(u)$ et qu'on a l'\'egalit\'e (\ref{E1}) dans $L$, on voit que
$u$ est repr\'esent\'e par  $<1,-a,-b,ab> $ 
sur le corps $L$. On en d\'eduit
$$(*)\;\; \omega(u)\mbox{ est pair.}$$

Puisque $\omega(p(-v))$ est pair d'apr\`es (\ref{ppair}), l'\'egalit\'e (\ref{E2}) donne:
$$(**)\;\; \omega(v)\mbox{ est pair.}$$

Comme on a $\omega(u+v) \geq {\rm inf}(\omega(u), \omega(v))$ avec \'egalit\'e si
$\omega(u) \neq \omega(v)$, de (*), (**) et l'hypoth\`ese (\ref{e}) on d\'eduit
$$\omega(u+v)> \omega(u)=\omega(v).$$

\smallskip

On a trois cas \`a consid\'erer:

\begin{enumerate}
\item 
Supposons $\omega(u)=\omega(v)=2n >0$. 
Alors on a $p(u) \in A^{\times}$, et la r\'eduction de $p(u)$ est l'image de $p(0)$ dans $\kappa$. Ainsi $p(u)/p(0)$ est un carr\'e $\rho^2$ dans le compl\'et\'e $\hat{L}$ de $L$.
On regarde l'\'egalit\'e (\ref{E1}) dans $L$:
$$x^2-ay^2-bz^2= u. p(u).$$
En divisant $x,y,z$  par  $\rho$, on en d\'eduit que
dans le compl\'et\'e $\hat{L}$ de $L$, il existe
$X,Y,Z \in \hat{L}$ 
avec
$$ X^2-aY^2-bZ^2= u. p(0).$$
De  $\omega(v)>0$   on d\'eduit que l'on a $p(-v) \in A^{\times} $ et 
  l'image de $p(-v)$ dans
 $\kappa$ est $p(0)$.  De (\ref{E2}) on d\'eduit alors que
 $v.p(0) $ s'\'ecrit comme le produit de $ab$ et d'un carr\'e dans $\hat{L}^{\times}$.
On conclut que $(u+v).p(0)$ est repr\'esent\'e par la forme $<1,-a,-b,ab>$
dans $\hat{L}$, et donc que l'on a $\partial(u+v,a,b)=0$.

\smallskip

\item Supposons $\omega(u)=\omega(v)=2n <0$.
Alors $p(u)$ est le produit de $u^{2n}$ et d'un \'el\'ement de $A^{\times}$ d'image 1 dans $\kappa$.
De m\^{e}me $p(-v)$    est le produit de $v^{2n}$ et d'un \'el\'ement de $A^{\times}$ d'image 1 dans $\kappa$.
Donc $p(u)$ et $p(-v)$ sont des carr\'es dans $\hat{L}^{\times}$.
De (\ref{E1}) on d\'eduit que $u$ est repr\'esent\'e par $<1,-a,-b>$ dans  $\hat{L}$
 et de (\ref{E2}) on d\'eduit que
$v/ab$ est un carr\'e dans  $\hat{L}^{\times}$. On en conclut que $u+v$ est repr\'esent\'e par
$<1,-a,-b,ab>$ sur $\hat{L}$, et donc que l'on a $\partial(u+v,a,b)=0$.

  \bigskip

\item Supposons   $\omega(u)=\omega(v)=0$. 
 Rappelons que l'on a une \'egalit\'e
$$p(u) = P(u)^2-a Q(u)^2 - bR(u)^2 + ab S(u)^2 \in k[u]$$
avec
$ \alpha(u) P(u) + \beta(u) Q(u) + \gamma(u) R(u) + \delta (u) S(u)=1.$

Comme on a $\omega(u)= 0$, on a $k[u] \subset A$ et un homomorphisme 
compos\'e d'anneaux $k[u] \to A \to \kappa$. 

L'image de $p(u)$ dans $\kappa$ est non nulle, sinon
la forme $<1,-a,-b,ab>$ aurait un z\'ero non trivial dans $\kappa$,
et  on aurait $(a,b)_{\kappa}=0$. Ainsi $\omega(p(u))=0$.
De m\^{e}me $\omega(p(-v))=0$.

De l'\'equation (\ref{E1}) on d\'eduit que l'unit\'e $u.p(u) \in A^{\times}$ est
repr\'esent\'ee par la forme $<1, -a,-b>$ dans $L$, et
donc aussi dans l'anneau de valuation discr\`ete $A$ par un argument connu 
(pour un \'enonc\'e plus g\'en\'eral, voir \cite[Thm. 4.5]{CLRR}).
Via la r\'eduction $A \to \kappa$, not\'ee $r  \mapsto \overline{r}$,
  on obtient une \'egalit\'e \footnote{On peut donner un argument plus rapide en passant dans le compl\'et\'e de $A$.}
$$  X^2-a Y^2 -b Z^2 = \overline{u}. p({\overline{u}}) \in \kappa^{\times}$$
avec $X,Y,Z \in   \kappa$.

 On a $\omega(v)=0$ et   $\omega(p(-v))= 0$.
Via (\ref{E2}),  on obtient $\omega(w) = 0$
et pour $w'=w/(ab)$ on a une \'egalit\'e
$$ab \overline{w'}^2=    \overline{v} p(-\overline{v}) \in \kappa^{\times}.$$
En additionnant les deux derni\`eres \'egalit\'es, on obtient
$$  X^2-a Y^2 -b Z^2 + ab \overline{w'}^2 = 
\overline{u} p({\overline{u}}) + \overline{v} p(-\overline{v}) \in \kappa$$
avec $X,Y,Z$ non tous nuls et aussi $\overline{w'} \neq 0$.
Comme on a $u,v \in A$, donc aussi $u+v \in A$,
si l'on suppose $\omega(u+v)$ impair, on a en particulier $\omega(u+v)>0$. 
On a donc $\overline{u}+\overline{v}=0$ et donc
$\overline{u}  p({\overline{u}}) + \overline{v}  p(-\overline{v})=0.$
Ceci donne une repr\'esentation non triviale dans $\kappa$
 $$ X^2-a Y^2 -b Z^2 + ab \overline{w'}^2 = 0,$$
 ce qui \'etait exclu.

 En r\'esum\'e, sous les hypoth\`eses  $\omega(u)=\omega(v)=0$,  
  on a  $\omega(u+v)$ pair,
  et donc $\partial(u+v,a,b)=0$.
 \end{enumerate}    
     
     \end{proof}

\subsection{Crit\`eres  pour la $CH_{0}$-trivialit\'e}\label{sectioncritere}
L'invariant $$(u+v,a,b)\in H^3_{nr}(k(X\times_k\Delta)/k,\Z/2)$$ de la section pr\'ec\'edente permet de d\'etecter si la vari\'et\'e $X$ est   universellement $CH_{0}$-triviale:

 \begin{thm}\label{maincritere}
 Dans la situation \ref{Situation 1}, on a :
 \begin{itemize}
 \item [(i)]
Soit $F$ un corps contenant $k$.
 Soit $m \in X(F) $ d'image $\pi (m) = c  \in \A^1(F) = F $ avec $c.p(c) \neq 0$.
La classe de $m-m_{\infty}$ dans $A_{0}(X_{F})$ est nulle 
  si et seulement si  $(c+v, a,b)    \in H^3_{nr}(F(\Delta)/k,\Z/2)$ est nul.
  \item [(ii)]  La classe
 $$(u+v, a,b)    \in H^3_{nr}(k(W)/k,\Z/2)$$ est nulle  si et seulement si
   pour tout corps $F$ contenant $k$ on a $A_{0}(X_{F})=0$.
  \end{itemize}
 \end{thm} 

 \begin{proof}
 L'\'enonc\'e $(i)$ a \'et\'e \'etabli au  lemme \ref{mainlemma}. 
   Rappelons que l'on a $m_{\infty}\in X(k)$, en particulier $X(k)\neq \emptyset$.
     
Montrons $(ii)$.  Soit
$m=\eta_X $ le point g\'en\'erique de $X$.
 L'image de $\eta_X $ par $\pi : X \to \P^1_{k}$ est
  le point g\'en\'erique $\eta$ de $\P^1$. 
  Passant sur le corps $k(X)$, l'image de $\eta_{X} \in X(k(X))$
  par $\pi$ est $u \in k(u)=k(\P^1)  \subset k(X)=\A^1(k(X)) \subset \P^1(k(X))$.
 D'apr\`es $(i)$,  la classe $\eta_{X} - m_{\infty}$ dans $A_{0}(X_{k(X)})$
est nulle si et seulement si la classe $(u+v, a,b)    \in H^3_{nr}(k(W)/k,\Z/2)$ 
est nulle. Comme rappel\'e dans l'introduction,  la classe
$\eta_{X} - m_{\infty}$ dans $A_{0}(X_{k(X)})$ est nulle si et seulement si 
 pour tout corps $F$ contenant $k$ on a $A_{0}(X_{F})=0$.
 \end{proof}

 On en d\'eduit des conditions \'equivalentes pour montrer que la vari\'et\'e $X$ est   universellement $CH_{0}$-triviale. Dans la suite, on va utiliser les crit\`eres $(c)$ et $(f)$.
 	  
\begin{thm}\label{equivalencesk}
Dans la situation \ref{Situation 1},
 les conditions suivantes sont \'equivalentes :
\begin{itemize}

\item [(a)] On a $A_{0}(X_{k(X)})=0$.
\item [(b)]  Pour tout corps $F$ contenant $k$, on a $A_{0}(X_{F})=0$.
 \item [(c)] Pour le corps $F=k(\Delta)$, l'application
$H^3(F,\Z/2) \to H^3_{nr}(F(X)/F,\Z/2)$  est un isomorphisme. 
 \item [(d)] 
 L'application $H^3_{nr}(k(\Delta)/k,\Z/2)  \to  H^3_{nr}(k(W)/k, \Z/2)$
 est un isomorphisme.
\item[(e)]  La somme $u+v \in k(W)$ des images r\'eciproques de
 $u \in k(\P^1) \subset k(X)$ et $v \in k(\Delta)$ dans $k(X \times_{k}\Delta)$
satisfait  l'\'equation  $(u+v,a, b)= 0 \in H^3(k(W),\Z/2)$.
\item [(f)] La fonction $(u+v) \in k(W)$ est repr\'esentable sur $k(W)$  par la forme quadratique
$$x^2-ay^2-bz^2+abt^2.$$
\end{itemize}
\end{thm}

\begin{proof}

D'apr\`es le th\'eor\`eme \ref{maincritere}, on obtient l'\'equivalence de $(a)$, $(b)$ et $(e)$. L'\'equivalence de $(e)$ et $(f)$ est donn\'ee par la proposition \ref{classique}.
Par ailleurs, 
comme rappel\'e dans l'introduction,
 l'\'enonc\'e $(a)$ implique $(c)$ (voir \cite{Merk}).

On va maintenant montrer: $(c)\Rightarrow (d)\Rightarrow (e)$.

 L'hypoth\`ese $(c)$ donne que l'application
$H^3(k(\Delta),\Z/2) \to H^3_{nr}(k(W)/k(\Delta), \Z/2)$ est un isomorphisme.
Comme on a $X(k) \neq \emptyset$, la projection $W=X\times_{k}\Delta \to \Delta$
admet une section. 
Comme la cohomologie non ramifi\'ee est fonctorielle contravariante
par morphismes de $k$-vari\'et\'es lisses connexes,
 on en d\'eduit que l'application
$H^3_{nr}(k(\Delta)/k,\Z/2) \to H^3_{nr}(k(W)/k, \Z/2)$ est un isomorphisme,
on a donc l'\'enonc\'e $(d)$.

Montrons que $(d)$ implique $(e)$.
D'apr\`es le  th\'eor\`eme \ref{criterenullitecycle}, la
  classe
  $(u+v,a,b ) \in H^3(k(W),\Z/2)$  appartient \`a $H^3_{nr}(k(W)/k, \Z/2)$.
    L'hypoth\`ese (d) montre donc que $(u+v,a,b)$ est dans l'image
  de $H^3_{nr}(k(\Delta)/k, \Z/2)$.  On peut \'evaluer la classe $(u+v,a,b)$ en le point
  $(x,y,z,u)=(0,0,0,0)$ de $X(k) \subset X(k(\Delta))$. Cela donne
  $(v,a,b) \in H^3(k(\Delta),\Z/2)$. L'\'equation de $\Delta$
  est $w^2=abvp(-v)$, et on a suppos\'e $$(p(u),a,b)=0 \in H^3(k(u),\Z/2).$$
  On obtient 
  $$(v,a,b)= (ab p(-v), a, b) = (ab, a, b)= (-a,a,b) + (-b,a,b)=0.$$
Ainsi $(d)$ implique $(e)$.

\end{proof}

\section{Le cas r\'eel}\label{sectionreel}

\subsection{Cohomologie non ramifi\'ee des courbes r\'eelles}
Les r\'esultats suivants pour les courbes sur  le corps des r\'eels sont   connus (Weichold, Witt, Geyer, voir  \cite{W1} \cite{G64} \cite{GH81}):

	\begin{prop}\label{ww} 
	Soit $G={\rm Gal}(\C/\R)$.
	Soit $\Gamma$ une $\R$-courbe projective, lisse, g\'eom\'etriquement connexe.
	Soit $S$ l'ensemble des composantes connexes de $\Gamma(\R)$.
	\begin{itemize}
	\item[(i)] Les  groupes $\Br(\Gamma)=H^2_{nr}(\R(\Gamma)/\R,  \Z/2)$ et $H^{i}_{nr}(\R(\Gamma)/\R,  \Z/2)$
	pour $i\geq 2$ sont isomorphes \`a $(\Z/2)^S$, l'application \'etant donn\'ee
	par l'\'evaluation sur chaque composante connexe.
	\item [(ii)] Soit $J_{\Gamma}$ la jacobienne de $\Gamma$.
	Supposons  $\Gamma(\R)$ connexe non vide. Alors $J_{\Gamma}(\R)$ est connexe,
	le groupe $J_{\Gamma}(\R)/2$ 
	et, pour $i \in \Z$, tous les groupes de cohomologie de Tate $\hat{H}^i(G,J_{\Gamma}(\C))$ sont nuls.
	L'application $\Z/2= \Br(\R) \to \Br(\Gamma)$ est un isomorphisme.
	\end{itemize}
	\end{prop}

\subsection{Les fibrations en quadriques  r\'eelles consid\'er\'ees}\label{situationR}

Soient $X$ et $\Delta$ comme dans la situation \ref{Situation 1}.
On s'int\'eresse maintenant au cas $k=\R$ et $a=b=-1$. 
 L'hypoth\`ese $(p(u),-1, -1)=0 \in H^3(\R(u),\Z/2)$
 \'equivaut au fait  que le polyn\^ome  s\'eparable $p(u)$ est une somme de 4 carr\'es,
 ou encore est positif  sur $\R$, ou encore est somme de 2 carr\'es dans $\R(u)$.
 Il \'equivaut aussi au fait   que $\Delta(\R)$  est connexe et non vide, et au fait que
  $X(\R)$ est  connexe  et non vide.
 On se place d\'esormais dans la situation :
  \begin{situation}\label{Situation 2}
  Soit 
 $p(u)\in \R[u]$ s\'eparable, unitaire, non constant, de degr\'e pair, avec $p(0)\neq 0$,
 positif sur $\R$.
Soit   $\pi: X \to \P^1_{\R}$ une fibration en surfaces quadriques donn\'ee
au-dessus de $\A^1_{\R}=\Spec(\R[u])$ par l'\'equation
$$ x^2+y^2+z^2 - u.p(u). t^2=0.$$ 
 Soit $\Delta$ la courbe  r\'eelle projective et lisse d'\'equation affine
 $$w^2= v.p(-v).$$  
Soit $W= X \times_{\R} \Delta$.
 \end{situation}

 Comme $\Delta$ est une courbe projective et lisse et $\Delta(\R)$  est connexe
non vide, la proposition \ref{ww} donne
 $H^i_{nr}(\R(\Delta)/\R,\Z/2)= H^i(\R,\Z/2) = \Z/2$ pour tout $i \geq 2$.
 
 L'\'equation affine $$ x^2+y^2+z^2 - u.p(u)=0$$
 donne que la fonction $u= (x^2+y^2+z^2)/p(u) \in \R(X)$
 est le quotient d'une somme de trois carr\'es par une somme de deux carr\'es,
 et donc  $u$ est une somme de 4 carr\'es dans $\R(X)$. De m\^eme,
 l'\'equation affine $$w^2= v.p(-v)$$ avec $p(-v)$ positif donne
 que $v \in \R(\Delta)$ est une somme de deux carr\'es.
 Ceci implique que $u+v \in \R(W)$  est une somme de 6 carr\'es.

	\subsection{Trivialit\'e de $H^3_{nr}(\R(X)/\R,\Z/2)$}\label{H3nrXR}
 
 Soient  $p(u) \in \R[u]$ et $X/\R$ comme dans la situation \ref{Situation 2}.
\begin{prop}\label{premierereduc}
Toute classe dans $H^3_{nr}(\R(X)/\R, \Z/2)$
est l'image d'un \'el\'ement de $H^3(\R(u),\Z/2)$
de la forme $(-1,-1,R(u)) \in H^3(\R(u),\Z/2)$, o\`u $R(u) \in \R[u]$
a toutes ses racines r\'eelles, distinctes,
et  strictement n\'egatives.
\end{prop}

\begin{proof}
Par la proposition \ref{KRSu},
toute classe $\alpha \in H^3_{nr}(\R(X)/\R, \Q/\Z(2))$
est image d'une classe dans $H^3(\R(u),\Q/\Z(2))$.
Comme $\C(u)$ est dimension cohomologique 1, 
toute classe  dans $H^3(\R(u),\Q/\Z(2))$ est d'exposant $2$.
En utilisant la proposition \ref{lemmaMerk} on obtient:
$$H^3(\R(u),\Z/2) = H^3(\R(u),\Q/\Z(2)).$$
 En utilisant que $(u,a-u)=0$ pour $a\in \R$ positif, que $(-1, r(u))=0$ si $r(u)$ est un polyn\^ome positif, 
 donc somme de deux carr\'es dans $\R[u]$,
 et la relation $(ab,c,d)=(a,c,d)+(b,c,d)$ pour les symboles, on voit que  tout \'el\'ement de $H^3(\R(u),\Z/2)$
  est   de la forme $(-1,-1,R(u))$, 
o\`u $R(u)$ est s\'eparable, avec toutes ses racines r\'eelles, notons-les $b_{i}$.
Comme $u$ est  une somme de 4 carr\'es dans $\R(X)$, on a
  $(-1,-1, u)= 0 \in H^3(\R(X),\Z/2)$. 
  On peut donc supposer tous les $b_{i} $ non nuls.

Le r\'esidu de $(-1,-1, \prod(u-b_{i})) \in H^3(\R(u),\Z/2) $ en $u=b_{i}$
est $(-1,-1)$. Si $\alpha \in H^3_{nr}(\R(X)/\R, \Q/\Z(2))$,
alors $(-1,-1)$ doit s'annuler dans $H^2(\R(X_{b_{i}}),\Z/2)$.
Si $b_{i}>0$ ceci n'est pas possible, car $X_{b_{i}}(\R) \neq \emptyset$.
Donc tous les $b_{i}$ sont n\'egatifs. 
\end{proof}

\begin{prop}\label{localglobal}
 Soit $p(u)$ comme dans la situation  \ref{Situation 2}.
 Soit $a\in \R, a>0$.
Il existe $\alpha(u), \beta(u) \in \R(u)$ et une \'egalit\'e
$$u+a = \alpha(u)^2 + u.p(u) \beta(u)^2 \in \R(u).$$
\end{prop}
 \begin{proof}
 L'\'enonc\'e \'equivaut \`a la trivialit\'e de la classe de l'alg\`ebre
 de quaternions $(-u.p(u), u+a) \in \Br(\R(u))$.
 Sur la droite affine sur $\R$, la suite exacte de Faddeev
(voir \cite[Prop. 1.5.1]{CTSkB})
s'\'ecrit
$$ 0 \to \Br(\R) \to \Br(\R(u)) \to \oplus_{x \in \A^1(\R)} \Z/2 \to 0,$$
o\`u les fl\`eches $\Br(\R(u)) \to H^1(\R,\Q/\Z)=\Z/2$ sont les r\'esidus.
On examine les  r\'esidus   de l'alg\`ebre
 de quaternions $(-u.p(u), u+a).$
 Au point $u=0$ (resp. $u=-a$) le  r\'esidu est $a \in \R^{\times}/\R^{\times 2}$ (resp. $a.p(-a)$), trivial car $a>0$ (resp. et $p(u)$ est positif sur $\R$).
Aux points ferm\'es non r\'eels, solutions de  $p(u)=0$, le r\'esidu est trivial car dans $\C^{\times}/\C^{\times 2}$.
Aux autres points de $\A^1_{\R}$, le r\'esidu est nul. Ainsi  $(-u.p(u),u+a)$ est dans l'image de $\Br(\R)$.     Par \'evaluation  en un point $u_{0}\in \R$ convenable, par exemple avec $u_{0}+a$ un carr\'e non nul,
on voit que la classe est nulle. 
 \end{proof}

\begin{prop}\label{quatre}
 Dans le corps des fonctions  $\R(X)$ de $X$, pour tout $a\geq 0$,   la fonction $u+a$
est une somme de $4$ carr\'es et $$(u+a, -1, -1) =0 \in H^3(\R(X),\Z/2).$$
\end{prop}
 
\begin{proof} 
On a vu que $u= (x^2+y^2+z^2)/p(u)$ est une somme de 4 carr\'es dans $\R(X)$.
Pour $a>0$, la proposition \ref{localglobal} donne 
$$ u+a =  \alpha(u)^2 + u.p(u) \beta(u)^2 =   \alpha(u)^2 + (x^2+y^2+z^2) \beta(u)^2 \in \R(X),$$
et donc $u+a$ est une somme de 4 carr\'es dans $\R(X)$.
\end{proof}

\begin{thm}\label{principal} Soit $X$ comme dans la situation \ref{Situation 2}.
L'application
 $$\Z/2=H^3(\R,\Z/2)  \to H^3_{nr}(\R(X)/\R, \Z/2)$$
est un isomorphisme.
\end{thm}
\begin{proof} On combine les propositions \ref{classique}, \ref{premierereduc} et \ref{quatre}. 
\end{proof}

\begin{rmk}\label{bonnefibquad}
Le r\'esultat ci-dessus peut aussi se d\'eduire d'un r\'esultat g\'en\'eral
qui se d\'eduit formellement de trois articles de O. Benoist et O. Wittenberg.
Soit $X \to \P^1_{\R}$ une fibration en surfaces quadriques de type (I), avec $X/\R$ projective et lisse. 
Soit $G={\rm Gal}(\C/\R)$.
  Alors le $G$-module galoisien $\Pic(X_{\C})$
 est de permutation. Supposons de plus $X(\R)$ connexe non vide.

 D'apr\`es \cite[Thm. 8.1]{BW20b}, la conjecture de Hodge r\'eelle enti\`ere
 vaut pour  les 1-cycles sur $X$. 

 D'apr\`es \cite[Thm 3.16]{BW20a}, pour $X$  
 comme  ci-dessus,  ceci implique  que
 $$CH_{1}(X) \to H_{1}(X(\R),\Z/2)$$
 est surjective.
 
 D'apr\`es le th\'eor\`eme 4.3 de \cite{BW}, on peut alors conclure
 que pour tout $G$-module $M$, et tout $i$,
 on a $H^{i}(G,M) \simeq H^{i}_{nr}(\R(X),M)$.
\end{rmk}

\subsection{Les crit\`eres de $CH_{0}$-trivialit\'e  dans le cas r\'eel}
 
 Le crit\`ere \ref{equivalencesk} se sp\'ecialise en l'\'enonc\'e suivant.
 
 \begin{thm}\label{equivalencesR}
Dans la situation \ref{Situation 2},
les conditions suivantes sont \'equivalentes :
\begin{itemize}
\item[(a)] On a $A_{0}(X_{\R(X)})=0$.
\item[(b)] Pour tout corps $F$ contenant $\R$, on a $A_{0}(X_{F})=0$.
\item[(c)] Pour le corps $F=\R(\Delta)$, l'application
$H^3(F,\Z/2) \to H^3_{nr}(F(X)/F,\Z/2)$  est un isomorphisme.
\item[(d)] 
L'application
$$\Z/2= H^3_{nr}(\R(\Delta)/\R,\Z/2)  \to H^3_{nr}(\R(W)/\R, \Z/2).$$
est un isomorphisme.
\item[(e)] La somme $u+v \in \R(W)$ des images r\'eciproques de
 $u \in \R(\P^1) \subset \R(X)$ et de $v \in \R(\Delta)$ dans $\R(X \times_{\R}\Delta)$
satisfait  l'\'equation  $$(u+v,-1, -1)= 0 \in H^3_{nr}(\R(W)/\R,\Z/2).$$
\item[(f)] La fonction $u+v \in \R(W)$ est une  somme de 4 carr\'es.
\end{itemize}
\end{thm}

\begin{rmk}
On a vu dans la section \ref{situationR} que $u+v \in \R(W)$ est une somme de  6 carr\'es.
\end{rmk}

\begin{rmk}
On verra plus loin (th\'eor\`eme \ref{finitude} et remarque \ref{remfinitude})
 que dans la situation \ref{Situation 2} le groupe
$H^3_{nr}(\R(W)/\R, \Z/2)$ est  fini.
\end{rmk}

\begin{rmk}
Soit $X$ comme ci-dessus.
 Par un r\'esultat g\'en\'eral sur les formes quadratiques en 4 variables (\cite[Chap. VII. Lemma 3.1]{Lam},  
  \cite[Prop. 7.2.4]{CTSkB}),
 la forme quadratique $<1,1,1,-u>$ en 4 variables 
a un z\'ero dans $F=\R(X)$ si et seulement si la forme $<1,1,1,1>$
 a un z\'ero dans
l'extension discriminant $E=F(\sqrt{-u})$, i.e. si et seulement si la classe
de l'alg\`ebre de quaternions $(-1,-1) \in \Br(\R)$ s'annule dans $\Br(E)$.

Posons $-u=w^2$.
Le corps $E$ est le corps des fonctions de la vari\'et\'e d'\'equation
$$x^2+y^2+z^2= -w^2.P(-w^2).$$
Comme  $P(u)$ est s\'eparable, non constant et non nul en $u=0$,
le polyn\^{o}me non constant $P(-w^2)$ a toutes ses racines simples sur $\C$
et donc n'est pas un carr\'e.
Le discriminant $w^2.P(-w^2)$ de la forme quadratique
$$<1,1,1, w^2.P(-w^2))>$$
sur le corps $\R(w)$ 
n'est  donc pas un carr\'e dans $\R(w)$ ni m\^eme dans $\C(w)$. Ceci implique que
  l'application $\Br(\R(w)) \to \Br(E)$ est injective \cite[Prop. 7.2.4]{CTSkB}.
  
Ainsi $\Br(\R) \to \Br(E)$ est injective. On conclut :

  {\it La fonction $u \in \R(X)$, 
qui est une somme de 4 carr\'es dans
$\R(X)$, n'est pas une somme de
3 carr\'es dans  $\R(X)$.}

Mais un tel \'enonc\'e ne permet pas de d\'ecider de la non rationalit\'e d'un
solide r\'eel. De fait $1+x^2+y^2+z^2$ n'est pas une somme de 3 carr\'es 
dans $\R(x,y,z)=\R(\P^3)$ (Cassels-Pfister) \cite[Chap. IX, Cor. 2.4]{Lam}.
\end{rmk}

\section{$CH_{0}$-trivialit\'e de    $X$   par  sommes de 4 carr\'es dans $\R(X\times_{\R}\Delta)$}\label{sectioncarres}

On va maintenant donner des exemples o\`u l'on peut v\'erifier la condition  $(f)$ du th\'eor\`eme \ref{equivalencesR} et l'on peut donc d\'emontrer que la vari\'et\'e $X$ est   universellement $CH_{0}$-triviale.

\begin{thm}\label{ex2}  
	Soit $p(u)=u^2+au+b\in \R[u]$ un polyn\^ome s\'eparable positif. Supposons  
	$$b\geq \frac{a^2}{3}.$$
	Dans la situation \ref{Situation 2}, on a :
\begin{enumerate}
\item  [(i)]
la fonction $$r(u,v)=\frac{up(u)+vp(-v)}{u+v}$$ est une somme de 3 carr\'es dans $\R(u,v)$;
\item [(ii)] la fonction $u+v$ est une somme de 4 carr\'es dans $\R(W)$;
\item [(iii)] pour tout corps $F$ contenant $\R$, on a $A_{0}(X_{F})=0$.
\end{enumerate}
	\end{thm}
	\begin{proof}
	Dans $\R(u,v)$, on a
	\begin{multline*}
	up(u)+vp(-v)=(u+v)(u^2-uv+v^2+au-av+b)=\\
	 =(u+v)\left(\left(u+\frac{a-v}{2}\right)^2 + \frac{3}{4}\left(v-\frac{a}{3}\right)^2+b-\frac{a^2}{3}\right). 
		\end{multline*}
		Par hypoth\`ese, on a  $b-\frac{a^2}{3}\geq 0$.
		Ainsi $$r(u,v)=\frac{up(u)+vp(-v)}{u+v}=\left(u+\frac{a-v}{2}\right)^2 + \frac{3}{4}\left(v-\frac{a}{3}\right)^2+b-\frac{a^2}{3}$$ est une somme de 3  carr\'es dans $\R(u,v)$, 
		et donc certainement une somme de 4 carr\'es.
		Dans le corps des fonctions $\R(W)$, on a 
		$$ x^2+y^2+z^2+w^2= up(u)+vp(-v)= (u+v).r(u,v).$$
		Comme le produit de deux sommes de 4 carr\'es dans un corps est une somme de 4 carr\'es 
		(Euler),
		et qu'il en est donc ainsi aussi pour le quotient de deux telles sommes, 
		on obtient que $u+v$ est une somme de 4 carr\'es dans $\R(W)$. Le dernier \'enonc\'e r\'esulte du th\'eor\`eme \ref{equivalencesR}.
	\end{proof}

\begin{rmk}	
\begin{itemize}
\item [(1)] Par exemple, la proposition ci-dessus s'applique \`a la vari\'et\'e $X$ donn\'ee par l'\'equation $x^2+y^2+z^2=u(u^2+1)$.
\item [(2)] Le polyn\^ome $p$ est positif si et seulement si $b>\frac{a^2}{4}$. Dans la proposition ci-dessus, on a besoin de la condition plus forte  $b\geq \frac{a^2}{3}$  (sauf pour $a=0$).
\end{itemize}
	\end{rmk}

En degr\'e sup\'erieur, nous avons les exemples suivants:

\begin{thm}\label{ex2n}
	Soit $n\geq 1$ un entier, et soit $$p(u)=u^{2n}+\sum_{i=0}^{n-1}a_{2i}u^{2i}
	\in \R[u]$$ un polyn\^ome s\'eparable positif, tel que $p(0)\neq 0$. Supposons de plus  
	$$a_{2i}\geq 0 \mbox{ pour tout } 0\leq i<n.$$
Dans la situation \ref{Situation 2}, on a :
\begin{enumerate}
 \item [(i)] La fonction $$r(u,v)=\frac{up(u)+vp(-v)}{u+v}$$ est une somme de 4 carr\'es dans $\R(u,v)$.
\item [(ii)] La fonction $u+v$ est une somme de 4 carr\'es dans $\R(W)$.
\item [(iii)] Pour tout corps $F$ contenant $\R$, on a $A_{0}(X_{F})=0$.
\end{enumerate}
\end{thm}

\begin{proof}
On fait le changement de variables $u=tv$ et on \'ecrit:
\begin{multline*}
r(tv,v)=\frac{tvp(tv)+vp(-v)}{tv+v}
=v^{2n}\frac{t^{2n+1}+1}{t+1}+\sum_{i=0}^{n-1}a_{2i}v^{2i}\frac{t^{2i+1}+1}{t+1}.
\end{multline*}

Notons que pour tout $i>0$ le polyn\^ome $\frac{t^{2i+1}+1}{t+1}$ a pour racines dans $\C$ des racines de l'unit\'e $\zeta$ avec $\zeta^{4i+2}=1$, distinctes de $\zeta=1$ et $\zeta=-1$, donc non r\'eelles.  Comme le coefficient dominant est $1$, pour toute valeur $t\in \R$ ce polyn\^ome ne prend que des valeurs positives sur $\R$. Puisque $a_{2i}\geq 0$ pour tout $i$ par hypoth\`ese, on en d\'eduit que le polyn\^ome $r(u,v)=r(tv,v)$ ne prend que des valeurs positives sur $\R$, c'est donc une somme de $4$ carr\'es dans $\R(u,v)$ d'apr\`es un th\'eor\`eme de Pfister (voir \cite[XI.4.11]{Lam}). 

Dans le corps des fonctions $\R(W)$ de la vari\'et\'e $W$, on a l'\'egalit\'e
$$(u+v)r(u,v)=x^2+y^2+z^2+w^2.$$
Ainsi, $u+v \in \R(W)$, 
qui est un quotient de deux sommes de 4 carr\'es dans $\R(W)$, 
 est une somme de 4 carr\'es.
 Le dernier \'enonc\'e r\'esulte du th\'eor\`eme \ref{equivalencesR}.
\end{proof}

\begin{rmk} Dans l'exemple ci-dessus, tous les coefficients de puissances impaires de $t$ sont nuls.
Si l'on rel\^{a}che cette hypoth\`ese, on peut au moins
 trouver un ouvert $U\subset \R^{2n+1}$ (pour la topologie euclidienne) dans l'espace  de coefficients de polyn\^omes  
$$p(u)=\sum_{i=0}^{2n} a_iu^{i},$$ tel que pour tout 
$(a_{2n}, a_{2n-1},\ldots, a_0)\in U$   la fonction 
$$r(u,v)=\frac{up(u)+vp(-v)}{u+v}$$  soit  positive sur $\R(u,v)$, et donc soit une somme de 4 carr\'es. 
Par \'evaluation en $v=0$, le polyn\^{o}me $p(u)$ est alors positif.
Par exemple, on peut d\'efinir $U$ comme suit.

Soit $0<i<n$.  Le polyn\^ome $$v^{2n}\frac{t^{2n+1}+1}{t+1}+v^{2i-1}\frac{t^{2i}-1}{t+1}\in \R[v,t]$$ est un polyn\^ome  de terme dominant $v^{2n}t^{2n}$, il est donc positif pour $v,t\geq T_i$ avec $T_i>0$ assez grand, et il est born\'e dans le compact $-T_i\leq v,t\leq T_i$. Il existe donc $m_i\in \R$ tel que $$v^{2n}\frac{t^{2n+1}+1}{t+1}+v^{2i-1}\frac{t^{2i}-1}{t+1}>m_i$$ pour tous $v,t\in \R$.
On prend $(a_{2n}, a_{2n-1},\ldots, a_0)$ dans l'ouvert $U\subset \R^{2n+1}$ d\'efini par les conditions
$$a_i\geq 0 \mbox{ pour tout }i,\; a_{2n}>\sum_{i=1}^{n} a_{2i-1}, \; a_0+\sum_{i=1}^{n-1}a_{2i-1}m_i>0.$$
On fait le changement de variables $u=vt$ et on \'ecrit:
\begin{multline*}
r(tv,v)=\frac{tvp(tv)+vp(-v)}{tv+v}=\\
=a_{2n}v^{2n}\frac{t^{2n+1}+1}{t+1}+\sum_{i=1}^{n-1}a_{2i-1}v^{2i-1}\frac{t^{2i}-1}{t+1}+\sum_{i=1}^{n-1}a_{2i}v^{2i}\frac{t^{2i+1}+1}{t+1}+a_0=\\
=(a_{2n}-\sum_{i=1}^{n} a_{2i-1})v^{2n}\frac{t^{2n+1}+1}{t+1}+\sum_{i=1}^{n-1}a_{2i-1}(v^{2n}\frac{t^{2n+1}+1}{t+1}+v^{2i-1}\frac{t^{2i}-1}{t+1}-m_i)+\\+\sum_{i=1}^{n-1}a_{2i}v^{2i}\frac{t^{2i+1}+1}{t+1}+(a_0+\sum_{i=1}^{n-1}a_{2i-1}m_i).
\end{multline*}
Comme dans la proposition ci-dessus, les fonctions $v^{2i}\frac{t^{2i+1}+1}{t+1}$ sont positives pour tout $i>0$. Tous les autres termes de la somme sont positifs par le choix de coefficients dans $U$.

En divisant par $a_{2n}$ on peut aussi supposer que $p$ est unitaire. La condition suppl\'ementaire que $p$ est s\'eparable est aussi ouverte.

\end{rmk}

\section{Lemmes de cohomologie galoisienne}\label{lemmescoho}
	Soit $K/k$ une extension finie galoisienne de corps, de groupe de Galois $G$.
	On note $N_{K/k}$  la norme de $K$ \`a $k$.
	On note $N_{G}=\sum_{g\in G} g \in \Z[G]$ et $I_{G} \subset \Z[G]$ 
	le noyau de l'augmentation $\epsilon : \Z[G] \to \Z$ envoyant tout $g\in G$ sur $1 \in\Z$.
	Pour $M$ un $G$-module, on utilise la cohomologie galoisienne de Tate.
	Notant ${}^{N_{G}}M$ le noyau de $N_{G} : M  \to M$, on sait que l'on a
	 $\hat{H}^{-1}(G,M) \simeq {}^{N_{G}}M/I_{G}M$ et $\hat{H}^0(G,M)= M^G/N_{G}M$.
	 Pour $G$ cyclique, et $i\in \Z$, on a des isomorphismes $\hat{H}^{i}(G,M) \simeq \hat{H}^{i+2}(G,M)$.

	\begin{lemma}\label{suiteHm1}
	Soit $X$ une $k$-vari\'et\'e projective et lisse g\'eom\'etriquement int\`egre, avec un $k$-point.
	On a des suites exactes
	\begin{equation} \label{s1}
	0 \to H^2(G,K^{\times}) \to \Ker[\Br(X) \to \Br(X_{K})] \to H^1(G,\Pic(X_{K})) \to 0
	\end{equation}
	et
	\begin{multline} \label{s2}
	1 \to k^{\times}/N_{K/k}K^{\times} \to \Ker[k(X)^{\times}/N_{K/k}(K(X)^{\times}) \to \Div(X)/N_{K/k}(\Div(X_{K}))]\to \\
	\to \hat{H}^{-1}(G,\Pic(X_{K}))  \to 0.
\end{multline}	
	Dans la suite (\ref{s2}), l'application de droite, qui est surjective, envoie la classe d'une fonction
$f\in k(X)^{\times}$ avec $ div_{X_{K}}(f)= N_{K/k}(\Delta)$ sur la classe de 
	$\Delta$ dans le groupe $ \hat{H}^{-1}(G,\Pic(X_{K})).$ 
       Lorsque $K/k$ est cyclique de degr\'e $n$, d\'efinie par un caract\`ere $\chi \in H^1(K/k,\Z/n)$,
	on passe de la seconde suite \`a la premi\`ere en prenant  le cup-produit avec $\chi$.
\end{lemma}

	\begin{proof} La suite exacte (\ref{s1})  est standard. Pour  obtenir la suite exacte  (\ref{s2})  on prend la cohomologie galoisienne de la suite exacte
	$$ 1 \to K(X)^{\times}/K^{\times} \to \Div(X_{K}) \to \Pic(X_{K}) \to 0.$$
	\end{proof}

	\begin{lemma}\label{Hm1J}
	Soit $K/k$ une extension de corps, finie galoisienne de groupe $G$.
	Soit $\Gamma$ une $k$-courbe projective, lisse, g\'eom\'etriquement int\`egre, poss\'edant
	un $k$-point. Soit $J_{\Gamma}= \Pic^0_{\Gamma/k}$ la jacobienne de $\Gamma$.
	Le degr\'e d\'efinit une suite exacte scind\'ee de $G$-modules
	$$ 0 \to J_{\Gamma}(K) \to \Pic(\Gamma_{K}) \to \Z  \to 0 .$$
	En particulier, on a 
	$$ H^1(G,  J_{\Gamma}(K)) \simeq H^1(G,  \Pic(\Gamma_{K}))$$
	et
	$$ \hat{H}^{-1}(G,  J_{\Gamma}(K)) \simeq  \hat{H}^{-1}(G,  \Pic(\Gamma_{K})).$$
	\end{lemma}

Le lemme suivant est bien connu.
\begin{lemma}\label{remG}
Soit $G=\Z/2=\{1,\sigma \}$.
Soit $\tilde{\Z}$ le r\'eseau  $\Z$ avec action antipodale de $\sigma$, et soit $\Z[G]$ le r\'eseau $\Z \oplus \Z.\sigma$.
 Tout $G$-r\'eseau $M$ est isomorphe comme $G$-module \`a une somme directe
$$ M \simeq \Z^{a} \oplus \tilde{\Z}^b \oplus \Z[G]^c,$$
o\`u les entiers $a,b,c$ sont uniquement d\'etermin\'es par $M$. 
On  a  :
\begin{enumerate}
\item  $\hat{H}^{i}(G,\Z[G])=0$ pour tout $i$.	
\item $\hat{H}^{i}(G,\Z)= \Z/2$, resp. $0$,  pour $i$ pair, resp. impair.
\item $\hat{H}^{i}(G,\tilde{\Z})= \Z/2$, resp. $0$,  pour $i$ impair, resp. pair.
\end{enumerate}
\end{lemma}

Dans ce qui suit on aura 
 besoin de l'\'enonc\'e suivant:

\begin{lemma}\label{h1mc}
Soit $\omega \in \C$ un entier quadratique imaginaire, satisfaisant une \'equation
$$\omega^2 - d \omega +c =0$$
avec $d,c \in \Z$, $d^2-4c <0$.
Soit $K=\Q(\omega)$.
Soit $G=\Z/2=\{1,\sigma\}$ le groupe de Galois  $\Gal(\C/\R)= \Gal(K/\Q)$.

Soit $M$ le $G$-module $\Z[\omega]$ avec l'action donn\'ee par
$$\sigma(\omega)= d -\omega.$$
Soit $M_1$ le $G$-module $\Z e_1\oplus \Z e_2$ avec l'action donn\'ee par $$\sigma(e_1)=-e_1, \; \sigma(e_2)=-de_1+e_2.$$  
\begin{itemize}
\item [(i)] Si $d$ est pair, alors $\hat{H}^{-1}(G,M)=\hat{H}^{-1}(G,M_1)=\Z/2$.
 \item [(ii)] Si $d$ est impair, alors $\hat{H}^{-1}(G,M)=\hat{H}^{-1}(G,M_1)=0$.
 \end{itemize}
\end{lemma}

\begin{proof}
\'Etant donn\'e un $G$-module $M$, notons $M^G=M^{+}$ le sous-module des invariants
et $M^{-}$ le sous-module des antiinvariants, form\'e des \'el\'ements $m$ avec $\sigma(m)=-m$.
Pour le module $M=\Z[\omega]$, on a $M^+=\Z \subset \Z[\omega]$.
Si $d$ est pair, $d=2r$, on v\'erifie que $\Z[\omega]^{-}= \Z[r+\omega]$, donc
$\Z[\omega]^+ \oplus \Z[\omega]^{-}= \Z[\omega]$. Ainsi $M \simeq \Z \oplus \tilde{\Z}$.
Si  $d$ est impair, $d=2r+1$, alors $\omega+\sigma(\omega)=2r+1$ donc $(\omega)-r + (\sigma(\omega)-r)=1$.
Notant $\rho=\omega-r$, on a   $\Z \rho  +  Z \sigma(\rho) = M$ et  donc $M \simeq \Z[G]$.
On conclut avec le lemme \ref{remG}.

Soit maintenant $M_1=\Z e_1\oplus \Z e_2$
 et soit $xe_1+ye_2\in M_1$.
On a alors:
$$(1-\sigma)(xe_1+ye_2)=xe_1+ye_2+xe_1+dye_1-ye_2=(2x+dy)e_1.$$
Par ailleurs,
$$N_{G}(xe_1+ye_2)=xe_1+ye_2-xe_1-dye_1+ye_2=y(-de_1+2e_2).$$
Ainsi la condition $N_{G}(xe_1+ye_2)=0$ est \'equivalente \`a la condition $y=0$ et on trouve ${}^{N_G}M_{1}=\Z e_1$.
On en d\'eduit que pour $d$ impair ${}^{N_G}M_{1}=I_GM_1$ et que pour $d$ pair ${}^{N_G}M_{1}/I_GM_1=\Z/2.$
\end{proof}

\section{Contr\^{o}le de $H^3_{nr}(F(X)/F, \Z/2)$ par le groupe de Brauer d'une tordue de $\Delta_{F}$}\label{fibconiques}
\begin{situation}\label{Situation 3}
Soit $X$ comme dans la situation \ref{Situation 2}. Un mod\`ele affine est donn\'e par
l'\'equation 
$$ x^2+y^2+z^2 - u.p(u)=0.$$ 
On consid\`ere un mod\`ele birationnel de $X$ comme 
{\it fibration en coniques sur le plan 
affine $\A^2_{\R}$}. 
On d\'efinit
$$\tilde{X} \subset \P^2_{\R} \times_{\R} \A^2_{\R}$$
  par l'\'equation
   $$x^2+y^2= (u.p(u)-z^2)t^2$$
   avec coordonn\'ees homog\`enes $(x,y,t)$ pour $\P^2_{\R}$
   et coordonn\'ees affines $(u,z)$ pour $ \A^2_{\R}$.
   La projection   $\tilde{X}\to\A^2_{\R} $ est propre, l'espace total est lisse.
    La fibration est ramifi\'ee  exactement  le long de la courbe affine lisse
   $\Gamma_{1} \subset \A^2_{\R}$ d'\'equation
    $$u.p(u)-z^2=0.$$
   Cette courbe admet une compactification lisse $\Gamma$
   qui est une courbe hyperelliptique (elliptique si le degr\'e de $p$ est 2). Le sch\'ema $\Gamma \setminus \Gamma_{1}$  est r\'eduit \`a un unique point r\'eel $O$.
   L'espace topologique  $\Gamma(\R)$ est connexe et non vide.
\end{situation}

\medskip

  \begin{thm}\label{BOH4}
  Dans la situation \ref{Situation 3}, 
   soit $F$ un corps contenant $\R$.
 Soit  $F'=F(\sqrt{-1})$.
  Soit $G= \Gal(F'/F)$.
  On a un isomorphisme
 $$H^3_{nr}(F(X)/F, \Z/2)/H^3(F,\Z/2) \simeq H^3_{nr}(F(X)/F, \Q/\Z(2))/H^3(F,\Q/\Z(2)).$$
Le groupe $H^3_{nr}(F(X)/F, \Q/\Z(2))/H^3(F,\Q/\Z(2))$ est un sous-quotient du
groupe  $ H^1(G,\Pic(\Gamma_{F'}))$.
  \end{thm}

 \begin{proof}
Si $F$ contient $\sqrt{-1}$, alors la $F$-vari\'et\'e $X_{F}$ est $F$-rationnelle.
 On suppose d\'esormais  $\sqrt{-1} \notin F$.
 
 La vari\'et\'e $X_{F}$ est lisse et on a $X(F) \neq \emptyset$.
 Soit $m \in X(F)$. Pour \'etablir l'isomorphisme, il suffit de montrer
 $$\tilde{H}^3_{nr}(F(X)/F, \Z/2) \simeq \tilde{H}^3_{nr}(F(X)/F, \Q/\Z(2)),$$
 pour les sous-groupes form\'es des \'el\'ements nuls au point $m$.
 La vari\'et\'e $X_{F'}$ est $F'$-rationnelle.
 On a donc   
 $\tilde{H}^3_{nr}(F'(X)/F', \Q/\Z(2))=0$
 (on normalise en un point de $X(F')$ au-dessus de $m$).
 Ceci implique (argument de norme) que $\tilde{H}^{3}_{nr}(F(X)/F, \Q/\Z(2))$  est annul\'e par 2,
 et donc \'egal \`a $\tilde{H}^3_{nr}(F(X)/F, \Z/2)$ par la proposition \ref{lemmaMerk}
 et une compatibilit\'e standard de r\'esidus.
 Ceci \'etablit l'isomorphisme.

 Soient $\tilde{X}$,  $\Gamma_{1} \subset \A^2_{\R}$ et  $\Gamma$  comme ci-dessus.
  On consid\`ere le diagramme commutatif suivant, o\`u $H^{n}(L,i) = H^n(L,\Q/\Z(i))$.

  $$\xymatrix{
  & H^3(F(\tilde{X}),2)/H^3(F,2) \ar[r]             & \oplus_{\xi} H^2(F(\xi_{1}),1) \oplus  H^2(F(\eta_{1}),1)            & \\
0 \ar[r] & H^3(F(\A^2),2)/H^3(F,2) \ar[u] \ar[r]  & \oplus_{\xi} H^2(F(\xi),1) \oplus  H^2(F(\eta),1)  \ar[u] \ar[r] &
 \oplus_{m}   
 H^1(F(m), 0) \\
& F^{\times}  \ar[u]  \ar[r] &  \  \  \  \  \ \ \ \ \   \ \ 0  \ \ \ \ \   \oplus \ \ \ H^2(F'/F,1)  \ar[u]  & }
     $$

Dans les deux lignes sup\'erieures, les fl\`eches horizontales sont des r\'esidus.
La fl\`eche $F^{\times} \to H^2(F'/F,1)$ envoie $c \in F^{\times}$ sur la classe
de l'alg\`ebre de quaternions $(c,-1)$. Tout \'el\'ement du noyau  $H^2(F'/F,1)$ de 
$H^2(F,1) \to H^2(F',1)$ est de cette forme.

La fl\`eche $F^{\times} \to H^3(F(\A^2),2)/H^3(F,2)$
 est d\'efinie 
par l'application $c \mapsto (-1,c,g)$. Les suites verticales sont des complexes.

La ligne horizontale m\'ediane est la suite de Bloch-Ogus (dans la version de Rost  \cite{Rost}), qui pour $\A^2_{F}$ est exacte
 \cite[Prop. 8.6]{Rost}. Les points $m$ sont les points ferm\'es de $\A^2_{F}$.
On note 
$$g(u,z)=u.P(u)-z^2 \in \R[u,z] \subset F[u,z].$$
Le point $\eta$ est le point g\'en\'erique de $\Gamma_{1}$,  courbe donn\'ee par l'\'equation 
$g(u,z)=0$. Les points $\xi \in \A^2_{F}$ sont les autres points de codimension 1.
Le point $\xi_{1}$ est le point de codimension 1 de $\tilde{X}$ au-dessus de $\xi$.
C'est le point g\'en\'erique d'une conique sur $F(\xi)$ de classe associ\'ee
$$(-1,g_{\xi}) \in H^2(F(\xi),\Z/2 ) \subset H^2(F(\xi),1)=\Br(F(\xi)),$$
 o\`u $g_{\xi}$ d\'esigne l'image de $g(u,z)$ dans
le corps r\'esiduel de $\xi$.
Le noyau de $H^2(F(\xi),1)  \to H^2(F(\xi_{1}),1)$
est engendr\'e par la classe
 $(-1,g_{\xi})$. Le point $\eta_{1}$ est donn\'ee par l'\'equation
$x^2+y^2=0$ sur le corps $F(\eta)$.  
Le noyau de $H^2(F(\eta),1) \to H^2(F(\eta_{1}), 1)$ est form\'e des
classes de quaternions  de la forme $(-1,h)$ avec $h \in F(\eta)$.

\medskip

Notons 
$H^3_{nr}(F(X)/\tilde{X},2)$
le sous-groupe de $H^3(F(X),2)$ form\'e des classes non ramifi\'ees aux points
de codimension 1 de $\tilde{X}$.

Un argument analogue \`a celui donn\'e ci-dessus donne un isomorphisme
$$H^3_{nr}(F(X)/\tilde{X}, \Z/2)/H^3(F,\Z/2) \simeq H^3_{nr}(F(X)/\tilde{X}, \Q/\Z(2))/H^3(F,\Q/\Z(2)).$$

Toute classe $$\rho \in \Ker[\Br(\Gamma_{F}) \to \Br(\Gamma_{F'})]$$
d\'efinit un \'el\'ement  $(0, \rho) \in \oplus_{\xi} H^2(F(\xi),1) \oplus  H^2(F(\eta),1)$
dont l'image dans le groupe $\oplus_{m} H^1(F(m), 0)$ est nulle. Comme la suite de Bloch-Ogus sur $\A^2_{F}$
est exacte, une telle classe est l'image d'un unique \'el\'ement $\alpha$ de
$H^3(F(\A^2),2)/H^3(F,2)$.
 L'image $\beta$ de cet \'el\'ement dans 
$H^3(F(\tilde{X}),2)/H^3(F,2)$ 
a ses r\'esidus nuls  en tous les points
de codimension 1 de $\tilde{X}$ au-dessus du point g\'en\'erique de $\A^2_{F}$
et en tous les points $\xi_{1}$ et en le point
$\eta_{1}$.  On a ainsi d\'efini un homomorphisme
$$ \Ker[\Br(\Gamma_{F}) \to \Br(\Gamma_{F'})] \to H^3_{nr}(F(X)/\tilde{X}, 2)/H^3(F,2).$$
La partie inf\'erieure du diagramme donne que cet homomorphisme
passe au quotient \`a gauche par  $\Ker[\Br(F) \to \Br(F')]$.
Comme on a $\Gamma(\R)\neq \emptyset$ et donc $\Gamma(F) \neq \emptyset$
le lemme \ref{suiteHm1} donne que
le quotient \`a gauche s'identifie \`a $H^1(G, \Pic(\Gamma_{F'}))$.
On a donc d\'efini un homomorphisme 
$$H^1(G, \Pic(\Gamma_{F'})) \to H^3_{nr}(F(X)/\tilde{X}, 2)/H^3(F,2).$$

Montrons que cet homomorphisme est surjectif.
Soit  $\beta \in H^3(F(\tilde{X}),2)$
 une classe non ramifi\'ee sur $\tilde{X}$.
Ses r\'esidus dans
$$ \oplus_{\xi} H^2(F(\xi_{1}),1) \oplus  H^2(F(\eta_{1}),1)$$
sont nuls.
Comme la classe $\beta$
est dans le groupe $H^3_{nr}(F(\tilde{X})/F(\A^2),2)$, 
elle est par la proposition  \ref{KRSu}\footnote{C'est ici que nous utilisons les coefficients $\Q/\Z(2)$
plut\^{o}t que $\Z/2$.}
  l'image d'une classe
$\alpha\in H^3(F(\A^2),2)$ par la fl\`eche 
$H^3(F(\A^2),2) \to H^3(F(\tilde{X}),2).$

Par la commutativit\'e du diagramme, le r\'esidu de $\alpha$ en un point $\xi$
est soit $0$ soit $(-1,g_{\xi})$. 
   Soit $P(u,z) \in F[u,z]$ le polyn\^{o}me sans facteur carr\'e d\'efinissant les 
   points de codimension 1 de $\A^2_{F}$ autres que $\eta$
   o\`u $\alpha$ a un r\'esidu non nul. Alors $$\alpha_{0} :=
   \alpha - (-1,g(u,z),P(u,z))  \in H^3(F(\A^2), \Z/2)$$ 
   a tous ses r\'esidus nuls aux points de codimension 1 de $\A^2_{F}$ autres que $\eta$.
   Comme on a $(-1,g(u,z))=0 \in H^2(F(\tilde{X}),1)$, l'image de $\alpha_{0}$
   dans $H^3(F(\tilde{X}),2)$ est $\beta$.

   Le r\'esidu de $\alpha_{0}$ en $\eta$ est dans le noyau de
   $H^2(F(\eta),1) \to H^2(F'(\eta),1)$, donc de la forme $(-1,h)$ avec
   $h\in F(\eta)$.
   Du complexe de Bloch-Ogus pour $\A^2_{F}$ r\'esulte que
le r\'esidu   $\partial_{\Gamma_{F} }(\alpha_{0})\in F(\Gamma_{1})$ 
  a tous ses r\'esidus $\partial_{m}\partial_{\Gamma_{F} }(\alpha_{0})$ nuls aux points
 ferm\'es $m$ de $\Gamma_{1,F}$.  Par la loi de r\'eciprocit\'e sur la courbe projective et lisse $\Gamma$, on a
 $$\sum_{m\in \Gamma_F^{(1)}} \partial_{m}\partial_{\Gamma_{F} }(\alpha_{0})=0.$$
 Comme la courbe   $\Gamma_{F}$   ne diff\`ere de $\Gamma_{1,F}$ que par un unique point $F$-rationnel, on conclut
 que     $\partial_{m}\partial_{\Gamma_{F} }(\alpha_{0})=0$ pour tout point ferm\'e $m \in \Gamma_{F}$.
On a donc  $\partial_{\Gamma_F}(\alpha_{0})  \in \Br(\Gamma_{F})[2]$, et  cet \'el\'ement s'annule par passage de $F$ \`a $F'$.
En conclusion $\beta$ est l'image d'un \'el\'ement $\alpha_{0} \in H^3(F(\A^2), 2)$ dont tous les r\'esidus sur $\A^2_{F}$ en un point autre que $\eta$ 
sont nuls et dont le r\'esidu en $\eta$ est une classe 
 $$\rho \in \Ker[\Br(\Gamma_{F}) \to \Br(\Gamma_{F'})].$$
 (Nous utilisons ici tacitement l'injectivit\'e du groupe de Brauer d'une vari\'et\'e lisse int\`egre dans le groupe de Brauer de 
 son corps des fonctions.)
 L'image de cette classe $\rho$ par l'homorphisme d\'efini ci-dessus est  la classe
 de $ \beta$ dans $H^3(F(\tilde{X}),2)/H^3(F,2).$
Ainsi  l'homorphisme
 $$H^1(G, \Pic(\Gamma_{F'})) \to H^3_{nr}(F(X)/\tilde{X}, 2)/H^3(F,2)$$
 est surjectif. 
 
 Comme on a l'inclusion
 $H^3_{nr}(F(X)/F,2) \subset H^3_{nr}(F(X)/\tilde{X},2)$,
 ceci \'etablit l'\'enonc\'e.
\end{proof}

\begin{rmk}
On doit pouvoir montrer que  l'inclusion
$$H^3_{nr}(F(X)/F,2) \subset H^3_{nr}(F(X)/\tilde{X},2)$$
est une \'egalit\'e, et donc que
le groupe $H^3_{nr}(F(X)/F, \Q/\Z(2))/H^3(F,\Q/\Z(2))$ est un quotient du
groupe  $ H^1(G,\Pic(\Gamma_{F'}))$, mais nous n'aurons pas besoin de cette pr\'ecision.
\end{rmk}

\begin{rmk}
La m\'ethode de la d\'emonstration du th\'eor\`eme \ref{BOH4}  n'est pas nouvelle.
Elle a \'et\'e classiquement appliqu\'ee \`a l'\'etude du groupe de Brauer non ramifi\'e
d'un espace total de fibration en quadriques. Elle a \'et\'e appliqu\'ee \`a l'\'etude
du troisi\`eme groupe de cohomologie non ramifi\'ee de telles fibrations
dans  \cite{PS16} et dans \cite[\S 5]{ACTP}.
\end{rmk}

En sp\'ecialisant le r\'esultat au cas $F=\R$, nous obtenons
une d\'emonstration alternative du th\'eor\`eme \ref{principal} :
\begin{cor}
Dans la situation  \ref{Situation 2},
 on a un isomorphisme $$\Z/2=H^3(\R,\Z/2) \simeq H^3_{nr}(\R(X)/\R, \Z/2)
= H^3_{nr}(\R(X)/\R, \Q/\Z(2)).$$
\end{cor}
\begin{proof} 
On a un isomorphisme
$$ H^1(G,\Pic(\Gamma_{\C}))= H^1(G, J_{\Gamma}(\C))$$
et ce groupe est nul car $\Gamma(\R)$ et donc $J_{\Gamma}(\R)$
sont connexes (proposition \ref{ww}). On applique le th\'eor\`eme \ref{BOH4}
avec $F=\R$.
\end{proof}

Nous nous int\'eressons maintenant au cas $F=\R(Y)$, o\`u $Y$
est une
 $\R$-vari\'et\'e projective,
  lisse, g\'eom\'etriquement int\`egre, munie d'un $\R$-point $m$.
Soit $\Gamma/\R$ la courbe  projective et lisse d'\'equation affine  $ z^2=u.p(u).$
 Soit  $F=\R(Y)$ et $F'=\C(Y)$. Soit $G=\Z/2=\{1,\sigma\}$ le groupe de Galois  $\Gal(\C/\R)$.
On pointe la courbe 
$\Gamma$ 
 par le point \`a l'infini (qui est r\'eel) du
rev\^etement double $\Gamma \to \P^1$ 
d\'efini par $u$. On note $J_{\Gamma}$, resp. $J_{Y}$,
la vari\'et\'e d'Albanese de $\Gamma$, resp. de $Y$.
On consid\`ere  les morphismes canoniques
 $ \Gamma \to J_{\Gamma}$, et $Y \to J_{Y}$
envoyant les points marqu\'es sur les points $0$ des vari\'et\'es d'Albanese.
 
Comme $Y$ est une vari\'et\'e lisse et $J_{\Gamma}$ est une vari\'et\'e ab\'elienne,
on a \cite[\S II,  no. 15, Th\'eor\`eme 6]{W48}
$$ J_{\Gamma}(F') =  Mor(Y_{\C}, J_{\Gamma}),$$
o\`u $Mor$ d\'esigne les morphismes de vari\'et\'es.
Via le choix des points marqu\'es, on a 
$$Mor(Y_{\C}, J_{\Gamma,\C}   )
 =  J_{\Gamma}(\C) \oplus Mor_{\bullet}(Y_{\C}, J_{\Gamma,\C}).$$

Le morphisme point\'e
  $Y \to  J_{Y}$ induit un isomorphisme $G$-\'equivariant
$$Mor_{\bullet}(Y_{\C}, J_{\Gamma,\C}) =  Hom_{varab}(J_{Y_{\C}}, J_{\Gamma,\C}).$$

On obtient donc un isomorphisme de $G$-modules
$$  J_{\Gamma}(F')  = J_{\Gamma}(\C) \oplus Hom_{varab}(J_{Y,\C}, 
J_{\Gamma,\C}).$$

Comme $\Gamma(\R)$ est connexe non vide, on
 a $H^1(G,J_{\Gamma}(\C))=0$ (proposition \ref{ww}).
 En utilisant le lemme \ref{Hm1J}, on obtient  
 \begin{equation}\label{h1picvshom}
 H^1(G, \Pic(\Gamma_{F'})) = H^1(G, Hom_{varab}(J_{Y,\C}    , J_{\Gamma,\C})).
 \end{equation}

\begin{thm}\label{finitude}
 Soit $X$ comme dans la situation \ref{Situation 2}.
Soit $Y$
une $\R$-vari\'et\'e projective, lisse, g\'eom\'etriquement int\`egre, avec $Y(\R) \neq \emptyset$.
Soit $F =\R(Y)$. Soit $W=X \times_{\R} Y$.
Si  $H^3_{nr}(\R(Y)/\R, \Q/\Z(2))$ est  fini,
alors le  groupe $H^3_{nr}(\R(W)/\R,\Q/\Z(2))$ est   fini.
 \end{thm}

\begin{proof} Le groupe ab\'elien $Hom_{varab}(J_{Y,\C}    , J_{\Gamma,\C})$
est un r\'eseau. Ainsi le groupe $H^1(G, \Pic(\Gamma_{F'}))$ est fini.  
Le th\'eor\`eme \ref{BOH4} donne que le quotient
$$H^3_{nr}(\R(W)/\R(Y), \Q/\Z(2))/ H^3(\R(Y), \Q/\Z(2))$$
est fini.
Comme on a $X(\R) \neq \emptyset$, la projection $W \to Y$ admet une section.
Ainsi d'une part l'application
$H^3(\R(Y),\Q/\Z(2)) \to H^3(\R(W),\Q/\Z(2)) $ 
est injective, d'autre part, par fonctorialit\'e contravariante de la cohomologie non ramifi\'ee
pour les morphismes de vari\'et\'es lisses,
 toute classe dans $$H^3_{nr}(\R(W)/\R,\Q/\Z(2)) \subset H^3(\R(W),\Q/\Z(2))$$
qui est dans l'image de $H^3(\R(Y),\Q/\Z(2))$ appartient \`a $H^3_{nr}(\R(Y)/\R,\Q/\Z(2)).$ 
Ce dernier groupe est  fini par hypoth\`ese.
\end{proof}

 \begin{rmk}\label{remfinitude}
Le groupe 
$H^3_{nr}(\R(Y)/\R, \Q/\Z(2))$ est fini pour  toute vari\'et\'e $Y$ projective, lisse,
g\'eom\'etriquement connexe de dimension au plus 2.  
 Il est en effet  alors annul\'e par $2$
car $H^3(\C(Y), \Q/\Z(2))=0$. Il s'identifie donc avec
$H^3_{nr}(\R(Y)/\R, \Z/2)$ par la proposition \ref{lemmaMerk}.
Et pour $Y$ de dimension au plus 2 ce dernier groupe est isomorphe \`a $(\Z/2)^s$, o\`u $s$
est le nombre de composantes connexes de $Y(\R)$ \cite{CTP}.
\end{rmk}

Revenons maintenant \`a la situation \ref{Situation 2}. Soit $\tilde{X} \to \A^2_{\R}$ comme ci-dessus.
On consid\`ere la courbe projective et lisse $\Delta$  d'\'equation affine $w^2=v.p(-v)$.
Soit $F=\R(\Delta)$.
Soit $g(u,z)=u.p(u)-z^2$. Soit
$\Gamma$ la $\R$-courbe projective et lisse d'\'equation affine
$g(u,z)=0$ dans $\A^2_{\R}$. 

Les r\'esidus de $(-1,u+v)$ sur la courbe $\Gamma_{F}$ sont tous
nuls : il suffit de le v\'erifier sur la courbe $\Gamma_{1,F}$ qui diff\`ere de $\Gamma_{F}$ par
un point rationnel. En $u=-v$ on a $-v.p(-v)-z^2=0$ et $v.p(-v)=w^2$, dans le corps r\'esiduel,
donc $-1$ est un carr\'e. On a donc $(-1,u+v) \in \Br(\Gamma_{F})$.
En le point $u=z=0$ de $\Gamma_{F} $, la classe $(-1,u+v)  \in \Br(\Gamma_{F})$
 prend la valeur $(-1, v) = (-1, w^2/p(-v))$, et comme $p(-v)$
 est une somme de deux carr\'es, $(-1,v)=0$.
\begin{thm}\label{si-1u+v}
 Dans la situation 
 \ref{Situation 2}, 
 notons $F=\R(\Delta)$.
  Si la classe  $(-1,u+v) \in \Br(F(\Gamma))/\Br(F)$
 est nulle, alors $$(u+v,-1,-1)=0 \in H^3(F(X),\Z/2)$$ et la $\R$-vari\'et\'e $X$
 est universellement $CH_{0}$-triviale.
\end{thm}

\begin{proof} 
Sous l'hypoth\`ese faite, les pr\'eliminaires impliquent
$(-1,u+v) = 0 \in \Br(\Gamma_{F})$.  
Soit comme ci-dessus $g(u,z)=u.p(u)-z^2$.
 Le cup-produit $$(-1,u+v,-g(u,z))=(-1,u+v, z^2-u.p(u))  \in 
H^3(F(\A^2), \Z/2) \subset H^3(F(\A^2), \Q/\Z(2))$$
a  alors tous ses r\'esidus nuls sur $\A^2_{F} $.
La suite de Bloch-Ogus donne que cette classe est constante.
Par \'evaluation  en tout point de $\A^2(F)$ avec $u=0$  et $z\neq 0$,
on a $(-1,u+v,-g(u,z))=0 \in H^3(F(\A^2), \Z/2)$.
L'image de $(-1,u+v,-g(u,z))= (-1,u+v, z^2-u.p(u))$ dans
 $H^3(F(X),\Z/2)$ est   
$$(-1,u+v, -x^2-y^2) = (-1,u+v,-1).$$
On conclut que $ (-1,u+v,-1)= 0 \in H^3(F(X),\Z/2)$. 
D'apr\`es le th\'eor\`eme \ref{equivalencesR},
cette derni\`ere condition est \'equivalente \`a la $CH_{0}$-trivialit\'e universelle de $X$.
\end{proof}

 \section{Cas sans multiplication complexe}\label{sansMC}

On garde les notations de la situation \ref{Situation 3}.
Ainsi  $Y=\Delta$
est la courbe projective et lisse  
d'\'equation affine $ w^2=v.p(-v)$
point\'ee par le point \`a l'infini (qui est r\'eel) 
du rev\^etement double de $\Delta \to \P^1$  donn\'e par $v$.
On a $F=\R(\Delta)$ et $F'=\C(\Delta)$.
La courbe $\Gamma/\R$ est la courbe  projective et lisse d'\'equation affine  $$ \Gamma_{1}: z^2=u.p(u).$$

\begin{prop}\label{h1pich}
Avec les notations ci-dessus, supposons que l'anneau des endomorphismes de $ J_{\Gamma,\C}$ est r\'eduit \`a $\Z$. Alors 
\begin{itemize}
\item [(i)]le groupe $\hat{H}^{-1}
(G,\Pic(\Gamma_{F'}))$ est isomorphe \`a $\Z/2$, engendr\'e par l'image de la fonction $u+v$ via la suite exacte (\ref{s2});
\item [(ii)]
le groupe
$\Ker[\Br(\Gamma_{F}) \to \Br(\Gamma_{F'}) ]$ est engendr\'e par $$\mbox{le groupe }\Ker[\Br(F)  \to \Br(F')]\mbox{ 
et la classe }(u+v,-1).$$
\end{itemize}
\end{prop}

\begin{proof}
D'apr\`es (\ref{h1picvshom}), on a
 $$H^1(G, \Pic(\Gamma_{F'})) = H^1(G, Hom_{varab}(J_{\Delta,\C}    , J_{\Gamma,\C})).$$

Sur le corps des complexes, on a l'isomorphisme de courbes point\'ees
$\Delta_{\C}  \simeq \Gamma_{\C}$ donn\'e par
$$ \theta : (v,w) \mapsto (-v,iw),$$
qui induit un isomorphisme de jacobiennes
$$ \phi_{\theta}: J_{\Delta,\C}   \simeq J_{\Gamma,\C}.$$

Puisque $\phi_{\theta}$ est un isomorphisme, on a donc un isomorphisme de groupes ab\'eliens abstraits:
$$\Z\simeq Hom_{varab}(  J_{\Gamma,\C},  
J_{\Gamma,\C})
\simeq  
Hom_{varab}
(J_{\Delta,\C},
  J_{\Gamma,\C})    $$
o\`u la premi\`ere fl\`eche envoie le g\'en\'erateur $1\in \mathbb Z$ sur l'application identit\'e dans 
$Hom_{varab}(  J_{\Gamma,\C},   J_{\Gamma,\C})$. 
 La deuxi\`eme fl\`eche envoie 
 $\rho\in Hom_{varab}(  J_{\Gamma,\C},   J_{\Gamma,\C})$ sur  $\rho\circ\phi_{\theta}.$

On a donc 
$$Hom_{varab}(J_{\Delta,\C},  J_{\Gamma,\C}) \simeq \Z\phi_{\theta}$$ avec l'action induite par l'action de $\theta$ sur les points.
Comme $$\sigma(\theta) : (v,w) \mapsto (-v,-iw)=-\theta,$$ 
 l'action de $G$ sur $Hom_{varab}(J_{\Delta,\C},  J_{\Gamma,\C}) \simeq \Z\phi_{\theta}$
est antipodale.
 Alors 
 $$\hat{H}^{-1}
(G,\Pic(\Gamma_{F'}))=H^1(G, \Pic(\Gamma_{F'})) = H^1(G, Hom_{varab}(J_{\Delta,\C}    , J_{\Gamma,\C})) =   \Z/2$$
(voir le lemme \ref{remG}).
 
 Via la suite exacte (\ref{s2}), on obtient que le g\'en\'erateur de $\hat{H}^{-1}(G, \Pic(\Gamma_{F'}))$ est donn\'e 
 par la classe de la fonction rationnelle $u+v$ dans $F(\Gamma)^{\times}/N(F'(\Gamma)^{\times})$. En effet, soit $$\theta(\eta)=\theta(v,w)=(-v,iw)\in \Gamma(F').$$ Alors 
 $\sigma(\theta(v,w))=(-v,-iw)$ o\`u $\sigma$ est la conjugaison complexe. Le diviseur de $u+v$
sur $\Gamma_{F'}$  est donn\'e par
$$\theta+\sigma(\theta)-2O=(\theta-O)+\sigma(\theta-O),$$ o\`u $O$ est le point \`a l'infini du rev\^etement double
 $\Gamma\to \mathbb P^1$. On en d\'eduit que le diviseur de la fonction $u+v$ est une norme.
Par construction, $u+v$ s'envoie donc sur le 
  g\'en\'erateur de $\hat{H}^{-1}
(G,\Pic(\Gamma_{F'}))$. 
 \end{proof}

Dans la situation du th\'eor\`eme \ref{equivalencesR}, avec 
 $F=\R(\Delta)$,  si la classe $(u+v,-1,-1)\in H^3(F(X), \Z/2)$ est nulle, alors 
 le groupe quotient $H^3_{nr}(F(X)/F, \Z/2)/H^3(F, \Z/2)$ est nul 
 et la vari\'et\'e $X$ est universellement $CH_{0}$-triviale.

 Lorsque la courbe $\Gamma$ n'a pas de multiplication complexe,
et $F=\R(\Delta)$,
on peut 
montrer que ce quotient $H^3_{nr}(F(X)/F, \Z/2)/H^3(F, \Z/2)$
est
engendr\'e par la classe $(u+v,-1,-1)$, et donc  est d'ordre au plus 2 :
 \begin{thm}\label{uniquegenerateur}
 Soient $p(u) \in \R[u], X, \Delta, \Gamma$ comme dans la situation \ref{Situation 3}.
   Soit $F=\R(\Delta)$.
    Supposons que  l'anneau des endomorphismes sur $\,\C$ de la jacobienne $ J_{\Gamma}$ de $\Gamma$  
    est r\'eduit \`a $\Z$: 
    $$Hom_{varab}(J_{\Gamma,\C},J_{\Gamma,\C})\simeq \mathbb Z.$$
   Alors le groupe $$H^3_{nr}(F(X)/F, \Z/2)/H^3(F, \Z/2)$$ 
    est  d'ordre au plus 2 et 
   est  engendr\'e par la classe $(u+v, -1, -1)$.
  \end{thm}
  
\begin{proof}
Sous l'hypoth\`ese que  l'anneau des endomorphismes de  $ J_{\Gamma}$ sur $\C$
est r\'eduit \`a $\Z$,  la proposition \ref{h1pich} assure que $$ \Ker[\Br(\Gamma_{F}) \to \Br(\Gamma_{F'})]/\Ker[\Br(F) \to \Br(F')]=\Z/2$$
est engendr\'e par  la classe de $(-1,u+v)$. 
Les th\'eor\`emes 
\ref{BOH4}
et 
\ref{si-1u+v}
 et leur d\'emonstration donnent une surjection
$$\Ker[\Br(\Gamma_{F}) \to \Br(\Gamma_{F'})]/\Ker[\Br(F) \to \Br(F')]  \twoheadrightarrow  H^3_{nr}(F(X)/\tilde{X}, \Z/2)/H^3(F, \Z/2)$$
qui envoie la classe de  $(-1,u+v) \in \Br(\Gamma_{F})$ sur la classe de $(-1,u+v,-1) \in H^3(F(X), \Z/2)$.
D'apr\`es le th\'eor\`eme  \ref{criterenullitecycle}, la classe  
 $(u+v, -1, -1) \in H^3(F(X),\Z/2)$  appartient \`a $H^3_{nr}(F(X)/F,\Z/2)$.
\end{proof}

\begin{rmk}  Soient $X$, $\Delta$  et $F=\R(\Delta)$  comme dans la situation \ref{Situation 2}. Soit $p(u)= u^2+au+b \in \R[u]$ avec $b > a^2/3$. 
Supposons que la courbe elliptique $\Gamma$ d'\'equation affine $z^2=u.p(u)$ n'a pas de multiplication complexe. 
Comme les invariants $j$ des courbes elliptiques avec multiplication complexe sont d\'enombrables, 
il y a de  nombreux tels exemples (pour le lien entre $\{a,b\}$ et l'invariant $j(\Gamma)$, voir le d\'ebut de la section 
\ref{scomp}).  On est donc dans la situation du th\'eor\`eme \ref{uniquegenerateur} ci-dessus.
L'hypoth\`ese sur $a$ et $b$ et le  th\'eor\`eme \ref{ex2}  donnent   $(-1,-1,u+v)=0 \in H^3(F(X),\Z/2)$.
La  proposition 8.1 donne $H^1(G,\Pic(\Gamma_{F'}))=\Z/2$ et  identifie le g\'en\'erateur.
Par les calculs dans la d\'emonstration des th\'eor\`emes \ref{BOH4}  et \ref{si-1u+v},
l'image de ce g\'en\'erateur dans $H^3(F(X),\Z/2)/H^3(F,\Z/2)$ est
$(-1,-1, u+v)$.  Pour tout tel exemple, l'application
$$ H^1(G,\Pic(\Gamma_{F'})) \to H^3_{nr}(F(X)/\tilde{X},2)/H^3(F,2)$$
du th\'eor\`eme \ref{BOH4} a un noyau non trivial. 
\end{rmk}

   L'hypoth\`ese
   de la proposition \ref{h1pich} et du th\'eor\`eme \ref{uniquegenerateur}
   est r\'ealis\'ee pour les courbes elliptiques (cas $deg(p)=2$) sans multiplication complexe. 
   Cette condition est aussi souvent r\'ealis\'ee pour les courbes hyperelliptiques (cas $deg(p)\geq 4$). En effet, on
  a  le th\'eor\`eme suivant \cite[Thm. 1.3]{Z00} :
   
  \begin{thm} (Zarhin)
Soit $K$ un corps de caract\'eristique z\'ero.
Soit $f(u) \in K[u]$ un polyn\^ome s\'eparable
irr\'eductible de degr\'e $n$  au moins \'egal \`a 5.
Si le groupe de Galois sur $K$ de l'\'equation $f(u)=0$
est le groupe sym\'etrique $S_{n}$,  alors
sur tout corps alg\'ebriquement clos contenant $K$,
 l'anneau des endomorphismes
de la jacobienne de la courbe  hyperelliptique
$$ y^2=f(u)$$ est r\'eduit \`a $\Z$.
\end{thm}

On en d\'eduit:

\begin{cor}
En tout degr\'e $d\geq 2$ tout  polyn\^{o}me r\'eel de degr\'e $2d+1$
de la forme $u.h(u)$ avec $h(u)$ unitaire 
de degr\'e $2d$
strictement
 positif 
 est limite de
tels polyn\^{o}mes  $h'$ avec la propri\'et\'e que
  l'anneau  des endomorphismes complexes
de la jacobienne de la courbe  hyperelliptique
$$ y^2=u.h'(u)$$ est r\'eduit \`a $\Z$.
\end{cor}

\begin{proof}

 Soit $n$ un entier, $n\geq 2$.
Partons de l'extension g\'en\'erique  du corps $\Q(z_{1}, \dots, z_{n})$
donn\'ee par l'\'equation
$$x^n + z_{1}x^{n-1} + \dots + z_{n} =0.$$
Le groupe de Galois absolu est $S_{n}$.
Par le th\'eor\`eme d'irr\'eductibilit\'e de Hilbert avec approximation
\cite{Ek},
on peut approcher tout point de $\R^n$ hors du lieu de
ramification par un  point de $\Q^n$ tel que le groupe
de Galois absolu de l'\'equation sp\'ecialis\'ee soit $S_{n}$.

Soit $h(u) \in \R[u]$ un polyn\^ome s\'eparable unitaire
positif sur $\R$, de degr\'e $2d$. 
On peut approcher le polyn\^{o}me $u.h(u) \in \R[u]$, de degr\'e $n=2d+1$,
par un polyn\^{o}me s\'eparable unitaire $r(u) \in \Q[u]$ de degr\'e $n$ de telle mani\`ere  que
le polyn\^{o}me $r(u)$ ait $S_{n}$ pour groupe de Galois absolu. 
  
 Plus pr\'ecis\'ement, si $V \subset \A^n_{\Q}$ est l'ouvert
 form\'e des points o\`u  le polyn\^{o}me 
$$x^n + z_{1}x^{n-1} + \dots + z_{n}$$
est s\'eparable, l'ensemble des points de $V(\Q)$
o\`u le polyn\^ome sp\'ecialis\'e a un groupe de Galois
absolu diff\'erent de $S_{n}$ est un ensemble mince.
Son compl\'ementaire dans $V(\Q) \subset \A_{n}(\Q)$
est en particulier Zariski dense.

Si le polyn\^{o}me unitaire $r(u) \in \Q[u]$ est suffisamment proche de $u.h(u)$, alors
il est  de la forme $(u-a).h'(u)$ avec  $a\in \R$ et $h'(u)$ unitaire positif sur $\R$
(lemme de Krasner, th\'eor\`eme des fonctions implicites).
\end{proof}

\section{Cas elliptique avec multiplication complexe impaire}\label{avecMCimpaire}

\begin{lemma}\label{lemmaemc}
 Soit $p(u)\in \R[u]$ un polyn\^ome s\'eparable unitaire de degr\'e $2$, avec $p(0)\neq 0$, tel que $p$ est
 positif sur $\R$.
    Soit $E\subset \mathbb P^2_{\R}$ la courbe elliptique
   $$ E: z^2=u.p(u)$$
   et soit 
   $\Delta\subset \P^2_{\R}$ la courbe elliptique $$\Delta: w^2=v.p(-v).$$
   Soient $F=\R(\Delta)$ et $F'=\C(\Delta)$. Soit $G=\Gal(\C/\R)=\{1,\sigma\}$.
    Supposons que  la courbe $E$ est \`a multiplication complexe, et que l'on a $$End_{\C}(E) =
\Z[\omega],$$  o\`u $\omega \in \C$ est un entier quadratique imaginaire, satisfaisant une \'equation
$$\omega^2 - d \omega +c =0, \mbox{ o\`u }d,c \in \Z\mbox{ et }d \mbox{ est {\bf impair}}.$$
   Alors 
    $$\hat{H}^{-1}(G,\Pic(E_{F'}))=0.$$
\end{lemma}
\begin{proof}
En utilisant le lemme \ref{Hm1J} et la proposition \ref{ww}, on obtient  
 $$ \hat{H}^{-1}
(G,\Pic(E_{F'}))=H^1(G, \Pic(E_{F'})) = H^1(G, Hom_{varab}({\Delta}_{\C}, E_{\C})).$$

Soit $\phi_{\theta}: \Delta_{\C}\to E_{\mathbb C}$ l'isomorphisme
 $$\phi_{\theta}(v,w)  =(-v, iw).$$
Cet isomorphisme induit un isomorphisme de groupes ab\'eliens 
$$End_{\C}(E)\to Hom_{varab}(\Delta_{\C}, E_{\C}),\; \rho\mapsto \rho\circ\phi_{\theta}.$$
 On a donc   $Hom_{varab}(\Delta_{\C}, E_{\C})=\Z \phi_{\theta}\oplus\Z (\omega\circ\phi_{\theta})$.

On examine l'action du groupe $G$:  
 on a $$\sigma(\phi_{\theta}(v,w))= \sigma(-v,iw)=(-v,-iw)=-\phi_\theta(v,w).$$   
 De m\^eme,  
 $$\sigma(\omega\circ\phi_{\theta}(v,w))= \sigma(\omega(-v,iw))\stackrel{(a)}{=}d(-v,-iw)-\omega(-v,-iw)=$$
 $$\stackrel{(b)}{=}-d\phi_{\theta}(v,w)+\omega(-v,iw)=-d\phi_{\theta}(v,w)+\omega\circ\phi_{\theta}(v,w).
 $$
    
L'\'egalit\'e $(a)$ utilise que le morphisme d'\'evaluation
$$End_{\C}(E_{\C})\times E_{\C}\to E_{\C}$$  est Galois-\'equivariant.
L'\'egalit\'e $(b)$ utilise que $\omega(-P)=-\omega(P)$ pour tout point $P$ de $E$ puisque $\omega$ est un endomorphisme de $E$.
  
 Notant $e_{1} =\phi_{\theta}$ et $e_{2}= \omega \circ \phi_{\theta}$, 
on voit 
 que le $G$-module $Hom_{varab}(\Delta_{\C} ,E_{\C})$ est isomorphe au $G$-module $M_1$ du lemme \ref{h1mc} $(ii)$. Ainsi $\hat{H}^{-1}(G,M_1)=0$.
\end{proof}

On peut maintenant donner une autre famille de  fibrations en surfaces quadriques $X\to \mathbb P^1_{\R}$
dont l'espace total $X$ est $CH_0$-universellement trivial.

 \begin{thm}\label{emc}
 Dans la situation \ref{Situation 2}, avec  $p(u)$ de degr\'e 2, supposons
 satisfaites toutes les hypoth\`eses du lemme  \ref{lemmaemc}.

   Alors 
   \begin{itemize}
\item[(i)] pour le corps $F=\R(\Delta)$, l'application
$H^3(F,\Z/2) \to H^3_{nr}(F(X)/F,\Z/2)$  est un isomorphisme; 
 \item[(ii)] pour tout corps $F$ contenant $\R$, on a $A_{0}(X_{F})=0$.
 \end{itemize}
  \end{thm}

\begin{proof}
Le th\'eor\`eme \ref{BOH4} et le lemme \ref{lemmaemc} donnent l'\'enonc\'e (i).
L'\'enonc\'e $(i)$ implique $(ii)$  d'apr\`es le th\'eor\`eme \ref{equivalencesR}. 
\end{proof}

\section{Courbes elliptiques satisfaisant les hypoth\`eses du th\'eor\`eme \ref{emc}}\label{zarhin}

Les r\'esultats de cette section nous ont \'et\'e communiqu\'es par Yuri  Zarhin.
Ils donnent un crit\`ere pour que l'hypoth\`ese 
 sur  $End_{\C}(E)$ dans le   th\'eor\`eme \ref{emc} ci-dessus soit satisfaite. 
 Voir  \cite{Z24} pour des résultats plus généraux.

 Si $D$ est un entier n\'egatif, on \'ecrit $\sqrt{D}=id$, o\`u $d\in \R$ est positif,  tel que $d^2=-D$.

\begin{prop} \label{propZ0} (voir aussi \cite[Definition 1.1]{Z24}) 
Soit $E$ une courbe elliptique sur $\mathbb R$ \`a multiplication complexe par un ordre dans un corps quadratique imaginaire $K=\mathbb Q(\sqrt{D})$,  o\`u $D$ est un entier n\'egatif, non divisible par un carr\'e.
Les conditions suivantes sont \'equivalentes:
\begin{itemize}
\item[(i)] $End_{\C}(E) =
\Z[\omega],$  o\`u $\omega \in \C$ est un entier quadratique imaginaire, satisfaisant une \'equation
$$\omega^2 - d \omega +c =0, \mbox{ o\`u }d,c \in \Z\mbox{ et }d \mbox{ est impair}.$$
\item[(ii)] $End_{\C}(E)=\mathbb Z\oplus f\mathcal O_K$ est un ordre dans le corps  $K$, o\`u $f$ est un entier impair, et $D\equiv 1\mbox{ mod }4$.
\item[(iii)] La restriction  de l'application trace $$K\to \Q, \;\lambda+\mu\sqrt{D}\mapsto 2\lambda$$ induit une application surjective $tr:End_{\C}(E)\to \Z$.
\end{itemize}
\end{prop}
\begin{proof}
Les conditions $(ii)$ et $(iii)$ sont \'equivalentes d'apr\`es la description de l'anneau des entiers du corps quadratique imaginaire $K$. La condition $(i)$ implique clairement $(iii)$. La condition $(ii)$ implique $(i)$: il suffit de prendre $d=f, c=f^2(1-D)/4$.
\end{proof}

\medskip

La proposition suivante est un cas particulier de \cite[Thm. 2.4]{Z24} (voir  \cite[Lemma 3.9, Ex. 3.10]{Z24}):
\begin{prop} \label{propZ}
Soit $E$ une courbe elliptique sur $\mathbb R$. Supposons que
le discriminant  d'une \'equation de Weierstra{\ss} de $E/\R$ est n\'egatif.
Alors $E(\mathbb C)\simeq E_{\tau}=\C/(\Z+\Z\tau)$ o\`u $\tau=\frac{1}{2}+i\frac{y}{2}$ avec $y\geq 1$ un nombre r\'eel. 

Le conditions \'equivalentes de la proposition \ref{propZ0} ci-dessus sont satisfaites pour $E$ si et seulement s'il existe un entier n\'egatif $D$, $D\equiv 1\mbox{ mod }4$, et des entiers positifs impairs $k, \beta$,  o\`u $\beta$ divise $k^2D$, tels que
\begin{equation}\label{taueq}
\tau=\frac{1}{2}+\frac{k}{2\beta}\sqrt{D}.
\end{equation}
Les invariants r\'eels $j(E_{\tau})$ pour $\tau$ de la forme (\ref{taueq}) sont dans l'intervalle $(-\infty, 1728]$, et ils sont  denses pour la topologie r\'eelle.
\end{prop}

\begin{rmk}
Pour une courbe elliptique $E$ sur $\R$ avec $j(E)\neq 0, 1728$ on  a l'\'equivalence entre les hypoth\`eses:
\begin{enumerate}
\item le discriminant  d'une \'equation de Weierstra{\ss} de $E/\R$ est n\'egatif;
\item $E(\R)$ est connexe;
\item $j(E)<1728$.
\end{enumerate}
(voir les formules \cite[III, $\S 1$]{S92}).
\end{rmk}

\begin{proof}

 Le premier \'enonc\'e est standard (voir \cite[\S III.1]{S92} et \cite[Thm. V.1.1c)]{S94} pour la d\'efinition du discriminant 
 d'une \'equation de Weierstra{\ss} d'une courbe elliptique $E$
 et \cite[Chapitre V, Prop. 2.1, preuves du Thm. 2.3 et Cor. 2.3.1]{S94} pour l'expression de $\tau$ et le fait que ${\rm sgn}(\Delta(E_{\tau}))={\rm sgn}(e^{2\pi i\tau})$).
 
 \medskip

 Supposons que $\tau$ soit de la forme (\ref{taueq}). Soit $\Lambda=\Z+\Z\tau$ le r\'eseau correspondant. Alors  
\begin{multline*} 
 End_{\C}(E_{\tau})=\{\alpha\in \C, \alpha\Lambda\subset \Lambda\}=\{\alpha\in \C, | \,\exists a,b,c,d\in \Z, \, \alpha= a+b\tau, \,\alpha\tau=c+d\tau\}=\\=\{\alpha\in \C, | \,\exists a,b,c,d\in \Z, \,\alpha=a+b/2+\frac{kb}{2\beta}\sqrt{D}, \alpha\tau = c+d\tau. \}
\end{multline*}

On \'ecrit cette derni\`ere condition $(a+b\tau)\tau=c+d\tau$
 explicitement:
$$ (a+b\tau)\tau = \frac a2+ \frac b4 +\frac{k^2bD}{4\beta^2}+\frac{k}{2\beta}\sqrt{D}(a+b)=c+d/2+\frac{kd}{2\beta}\sqrt{D},$$
ce qui est \'equivalent \`a la condition
$$\; d=a+b,\; c=\frac{b}{4}\left(\frac{k^2D-\beta^2}{\beta^2}\right)$$. 

Ainsi $$End_{\C}(E_{\tau})=\{\alpha\in \C, | \,\exists a,b\in \Z, \, \alpha=a+b/2+\frac{kb}{2\beta}\sqrt{D}\mbox{ et } \frac{b}{4}\left(\frac{k^2D-\beta^2}{\beta^2}\right) \mbox{est dans }\Z.\}
$$

Comme on a suppos\'e que $\tau$ est de la forme (\ref{taueq}), on a que $D\equiv 1\mbox{ mod }4$,  $k, \beta$ sont impairs et $\beta | k^2D$,  d'o\`u
$4\beta|k^2D-\beta^2$.
Ainsi  pour tout $a$, et pour $b$ divisible par $\beta$, la condition  $\frac{b}{4}\left(\frac{k^2D-\beta^2}{\beta^2}\right)\in \Z$ est satisfaite. Soit alors $b=\beta n$ o\`u $n\in \Z$ impair. 
On a alors $tr(\alpha)=tr(a+b/2+\frac{kb}{2\beta}\sqrt{D})=2a+b$,
ainsi la condition (iii) de la proposition \ref{propZ0} est satisfaite.
 
\medskip

Inversement, on \'ecrit les conditions $\alpha= a+b\tau, \,\alpha\tau=c+d\tau$ pour $\tau=\frac{1}{2}+i\frac{y}{2}=\frac{1}{2}+\frac{1}{2}\sqrt{t}$ o\`u $t=-y^2$ est un nombre r\'eel n\'egatif. On obtient:
$$\alpha=a+b/2+\frac{b}{2}\sqrt{t},\; \alpha\tau = \frac a2+\frac b4 +\frac{bt}{4}+\frac 12 \sqrt{t}(a+b)=c+d\tau,$$
d'o\`u
$$d=a+b, \;c=\frac{b}{4}(t-1).$$
Ainsi $t$ est rationnel, on peut l'\'ecrire $t=\frac{m}{\beta}=\frac{m\beta}{\beta^2}=\frac{k^2D}{\beta^2}$,
o\`u $m,\beta, k,D$ sont des entiers, $m\beta=k^2D<0$,
le carr\'e
  $k^2$ est 
le plus grand carr\'e entier qui
 divise $m\beta$, et $D$ n'est pas divisible par un carr\'e. Ainsi $D<0$ et $\beta|k^2D$ par d\'efinition.
La condition $c=\frac{b}{4}(t-1)=\frac{b}{4}\left(\frac{k^2D-\beta^2}{\beta^2}\right)\in \mathbb Z$ et la condition $(iii)$ de la proposition \ref{propZ0} donnent que $b$ prend des valeurs impaires, et donc $4$ divise $k^2D-\beta^2$, d'o\`u l'on d\'eduit $D\equiv 1$ mod $4$, et $k,\beta$ sont impairs.
 
\medskip

  Finalement,  soit $f:[\frac{1}{2}, \infty)\to (-\infty, \infty)$
   la fonction r\'eelle d\'efinie par $$f(y)=j( \C/(\Z+\Z(\frac{1}{2}+iy))).$$
 On a:
 \begin{itemize}
  \item[(i)] La fonction $f$   induit une bijection continue   
  d\'ecroissante
  $$ f : [1/2, \infty) \to [1728, -\infty).$$ On a  $f(1/2)=1728$, $f(\sqrt{3}/2)=0$.
 Pour ceci, voir \cite[Chapitre V, Prop. 2.1, Thm. 2.3  et leurs preuves]{S94}. 
 \item[(ii)] Soient $n$ un entier impair, $\beta=n^2+2, D=-(n^2+2)$, et $k>n$ un entier impair. Soit 
 $$y_{k,n}=\frac{k}{2(n^2+2)}\sqrt{n^2+2}=\frac{k}{2\sqrt{n^2+2}}$$ et soit $\tau_{k,n}=1/2+iy_{k,n}$.
 D'apr\`es ce qui pr\'ec\`ede, les conditions de la proposition \ref{propZ0} sont satisfaites pour les courbes elliptiques $E_{\tau_{k,n}}$. Par ailleurs, les r\'eels $y_{k,n}$ sont denses dans l'intervalle $[\frac{1}{2}, \infty)$. Puisque $f$ est continue et bijective, les $j$-invariants de $E_{\tau_{k,n}}$ sont denses dans $(-\infty, 1728]$.
 \end{itemize}
\end{proof}

\section{Comparaison des r\'esultats des deux m\'ethodes
pour $p(u)$ de degr\'e 2}\label{scomp}

Soit $p(u)=u^2+au+b \in \R[u]$ un polyn\^ome s\'eparable unitaire de degr\'e $2$, avec $b=p(0)\neq 0$, tel que $p$ est
 positif sur $\R$. On a donc 
 $b>0$ et
 $0 \leq a^2/b <4$.
   Soit $X$ le mod\`ele projectif et lisse standard de la vari\'et\'e d'\'equation
   $$ x^2+y^2+z^2= u.p(u),$$
On compare ici  les r\'esultats de $CH_{0} $-trivialit\'e sur $X$  obtenus   \`a la section \ref{sectioncarres}
 (premi\`ere m\'ethode, sommes de carr\'es) avec ceux obtenus \`a la section  \ref{avecMCimpaire} (deuxi\`eme m\'ethode, multiplication complexe impaire).

Les formules \cite[III, $\S 1$]{S92} pour le $j$-invariant de la courbe $$E: z^2=u^3+au^2+bu$$ donnent
$$\Delta(E)=-16b^3(4-a^2/b)<0, \mbox{ et}$$
$$j(E)=256\frac{(3-(a^2/b))^3}{4-(a^2/b)}.$$
L'invariant $j(E)$ est un nombre r\'eel.
On a  $0 \leq a^2/b \leq 3$  si et seulement si $j(E)  \geq 0$,
et alors $0 \leq j(E) \leq 1728$.
On a $3 \leq a^2/b < 4$ si et seulement si $j(E)\leq  0$.
On a $a^2/b=3$ si et seulement si $j(E)=0$.
Pour $a^2/b=0$, i.e. $a=0$ on a $j(E)=1728$.

\bigskip

Soit $X/\R$ d'\'equation affine 
\begin{equation}\label{lasteq}
X:\;x^2+y^2+z^2=u.p(u).
\end{equation}

On a montr\'e  que le groupe $A_0(X_F)$ est nul pour tout corps 
$F$ contenant $\R$ dans chacun des cas suivants.
 
$(*)$  (premi\`ere m\'ethode) 
 D'apr\`es le th\'eor\`eme \ref{ex2}
 il en est ainsi  si $0 \leq a^2/b \leq 3$, c'est-\`a-dire si  $j(E)\geq 0$.
Dans ce cas $ j(E)$  
d\'ecro\^{\i}t  
contin\^{u}ment de
  $1728$ \`a $0$.

$(**)$  (deuxi\`eme m\'ethode)
 D'apr\`es le th\'eor\`eme \ref{emc},  il en est ainsi  si la courbe elliptique $E/\R$ est \`a multiplication complexe et l'on a $End_{\C}(E) = \Z[\omega],$  o\`u $\omega \in \C$ est un entier quadratique imaginaire, satisfaisant une \'equation $\omega^2 - d \omega +c =0$ avec $d,c \in \Z$, et $d$ impair.
Dans ce cas les invariants $j(E)$ sont alg\'ebriques et en particulier d\'enombrables.
Pour $j(E) <0$ seule cette m\'ethode est \'eventuellement disponible.
 D'apr\`es la proposition \ref{propZ} (Zarhin),
 on obtient que les invariants r\'eels $j(E)$ correspondant \`a ce cas sont dans l'intervalle $[-\infty, 1728]$, et ils sont  denses pour la topologie usuelle dans 
cet  intervalle.

Voici quelques exemples.

\begin{enumerate}

\item  Soit
$p(u)=u^2-3u+3$.
La courbe elliptique  $E$ est donn\'ee par l'\'equation
$$z^2=(u-1)^3+1.$$
On a $j(E)=0$.
La courbe $E$ est \`a multiplication complexe par les racines cubiques de l'unit\'e.
L'\'equation est ici $\omega^2 + \omega +1 =0$ et
$d=-1$ 
 est impair.
 Les deux m\'ethodes s'appliquent.

\item 

Soit $p(u)=u(u^2+1)$. La courbe elliptique $E$ est donn\'ee par l'\'equation
$$z^2=u(u^2+1).$$
On a $j(E)=1728.$
Dans ce cas la premi\`ere m\'ethode s'applique.
La seconde m\'ethode ne s'applique pas. On a multiplication complexe par $\Z[\sqrt{-1}]$. 
L'\'equation  est ici $\omega^2+1=0$ et $d=0$ n'est pas impair.

\item Soit  $p(u)=u^2-21u+112.$
La courbe elliptique $E$  est donn\'ee par l'\'equation
$$z^2=u(u^2-21u+112).$$
On a ici  $a=-21$, $b=112$, $p$ est un polyn\^ome positif avec $3< \frac{a^2}{b}$ donc $j(E)<0$.
La premi\`ere m\'ethode ne s'applique pas.
 Apr\`es le changement de variables $t=u+7$ on obtient que $E$ est donn\'ee par l'\'equation 
$z^2=t^3-35t+98$ que l'on trouve dans la base de donn\'ees \cite{LMFDB}.
Elle  a multiplication complexe par $\Z[\omega]$ avec $\omega = \frac{(1+\sqrt{-7})}{2}$.
L'\'equation  est ici $\omega^2-\omega+2=0$ et
$d=1$
est impair.
La seconde m\'ethode s'applique.
\end{enumerate}

\section{Rationalit\'e}\label{sectioncomp}

On donne des exemples de fibrations en quadriques $$\pi: X\to \P^1_{\R}$$ 
de dimension relative
au moins 1
dont l'espace total
est rationnel sur $\R$.

\subsection{Th\'eor\`eme de Witt \cite{W1,W2})}
Dans \cite[Satz 22]{W2}, Witt montre :
\begin{thm} Soit $\Gamma$ une courbe r\'eelle g\'eom\'etriquement int\`egre.
Une forme quadratique non d\'eg\'en\'er\'ee en au moins 3 variables sur
le corps $\R(\Gamma) $  repr\'esente z\'ero si elle  repr\'esente
z\'ero sur  toutes les compl\'etions de $\R(\Gamma)$ sauf au plus
un nombre fini.
\end{thm}

En particulier, si $\Gamma=\P^1_{\R}$ et
l'application $X(\R)\to \P^1_{\R}(\R)$ est surjective, alors la fibration $\pi$ admet une section, et donc $X$ est rationnelle.

\subsection{Fibrations en coniques sur $\P^1_{\R}$}

Soit maintenant $X  \to \P^1_{\R}$ d\'efinie par une \'equation affine
$$x^2+a(u) y^2 = u.f(u)$$
avec $a(u)$ et $f(u)$ deux polyn\^omes positifs sur $\R$
premiers entre eux.
Dans ce cas $X(\R)$ est connexe et non vide.

Pour  le corps $\R(u)$, 
 la suite exacte de Faddeev  \cite[Thm. 1.5.2]{CTSkB}
 s'\'ecrit
 $$ 0 \to \Br(\R) \to \Br(\R(u)) \to \oplus_{P \in \P^1(\R)}  H^1(\R,\Z/2) \to \Z/2 \to 0$$
 soit encore
  $$ 0 \to \Br(\R) \to \Br(\R(u)) \to \oplus_{P \in \P^1(\R)}  \Z/2 \to \Z/2 \to 0.$$
L'alg\`ebre de quaternions $(-a(u),-uf(u)) \in \Br(\R(u))$
a  deux r\'esidus non nuls, en $u=0$ et $u=\infty$.
Il en est de m\^{e}me de l'alg\`ebre $(-1,u)$.
Ces deux alg\`ebres ont donc m\^{e}me classe dans $\Br(\R(u))$  
 \`a addition pr\`es d'un
\'el\'ement de $\Br(\R)$. Par \'evaluation en un point r\'eel avec
$u>0$, on voit qu'elles sont \'egales.

Les  formes quadratiques de rang 3  sur $\R(u)$
$$<1,a(u), -uf(u)>$$ et $$ a(u) f(u) <1, 1,-u>$$
ont donc m\^eme discriminant et m\^eme invariant
de Hasse-Witt.   Elles sont donc isomorphes sur $\R(u)$.
Ainsi la fibre g\'en\'erique de $X/\P^1_{\R}$ est la conique
d'\'equation affine  $r^2+s^2-u=0$. Son corps des fonctions
est transcendant pur sur $\R$. Le corps des fonctions
de $X$ est donc transcendant pur sur $\R$.

Plus g\'en\'eralement, soit $X/ \P^1_{\R}$ une fibration en coniques
relativemement minimale. Soit
 $A/\R(u)$
  l'alg\`ebre de quaternions associ\'ee
\`a la fibre g\'en\'erique. Les r\'esidus non nuls de cette alg\`ebre en des points de $\P^1(\R)$
sont en nombre pair et
correspondent
\`a des fibres singuli\`eres form\'ees de deux droites conjugu\'ees, il y a des points r\'eels
d'un c\^{o}t\'e d'une telle fibre, et pas de l'autre. Si $X(\R)$ est connexe, il y donc
z\'ero ou deux tels points. Dans le second cas, on peut supposer que ces points correspondent
\`a $u=0$ et $u=\infty$. Dans le premier cas, l'alg\`ebre $A$ est constante, et donc nulle dans $\Br(\R)$
sinon on aurait $X(\R)=\emptyset$. Ceci implique que $X$ est birationnelle \`a $\P^1_{\R} \times_{\R} \P^1_{\R}$.
Dans le second cas,  la classe de $A \in \Br(\R(u))$ s'\'ecrit $(-1,u)$ (quitte \`a changer $u$ en $-u)$.
Ainsi la fibre g\'en\'erique de $X/\P^1$ est $\R(u)$-isomorphe \`a $X^2+Y^2-uT^2=0$
dans $\P^2_{\R(u)}$. Donc $X$ est $\R$-rationnelle, car $\R$-birationnelle \`a la surface affine
d'\'equation $x^2+y^2-u=0$.

On dispose d'un r\'esultat plus g\'en\'eral :

\begin{thm}
Soit $X/\R$ une surface projective et lisse g\'eom\'etriquement rationnelle
poss\'edant un point r\'eel.
Les conditions suivantes sont \'equivalentes :
\begin{itemize}
\item[(i)] L'espace topologique $X(\R)$ est connexe non vide.
\item[(ii)] La $\R$-surface $X$ est  rationnelle.
\item[(iii)] La $\R$-surface $X$ est stablement rationnelle.
\item[(iv)] La $\R$-surface $X$ est r\'etractilement rationnelle.
\item[(v)] La $\R$-surface $X$ est universellement $CH_{0}$-triviale.
\item[(vi)] Le groupe $A_{0}(X)$ est nul.
\end{itemize}
\end{thm}
\begin{proof}
Que  $(i)$ implique $(ii)$ 
remonte \`a Comessatti 
\cite{Com}. 
Une d\'emonstration  est donn\'ee par
 R. Silhol  \cite[Cor. 6.5]{Silhol}. 
La d\'emonstration repose sur la classification $k$-birationnelle des $k$-surfaces
 g\'eom\'e\-tri\-quement rationnelles.
 Que  $(vi)$ implique $(i)$ est un cas particulier de \cite{CTIschebeck}.
\end{proof}

\subsection{Un exemple en dimension sup\'erieure}
Soit $p(u)\in \R[u]$ non nul, positif sur $\R$, donc de
la forme $p(u)=a(u)^2+b(u)^2$ avec $a(u), b(u) \in \R[u]$.
Soit $m\geq 1$. Toute vari\'et\'e sur $\R$  d\'efinie par une \'equation
$$ \sum_{j=1}^{m} (x_{j}^2+x_{j+1}^2)= u p(u)$$
est rationnelle sur $\R$.
En effet,   le changement de variables
$$ X_{j} +   X_{j+1} \sqrt{-1}  = (x_{j} +  x_{j+1} \sqrt{-1})/(a(u)+  b(u) \sqrt{-1})$$
montre que cette vari\'et\'e est $\R$-birationnelle \`a celle d\'efinie
par
$$\sum_{j=1}^{m} (X_{j}^2+X_{j+1}^2)=u,$$
qui est   $\R$-isomorphe \`a l'espace affine $\A^{2m}_{\R}$.

\bigskip

\section{Vari\'et\'es non rationnelles : deux m\'ethodes}\label{deuxnonrat}

\subsection{Fibrations non de type (I) et calcul de groupe de Brauer}\label{contreexemple}

Soit $X \to \P^1_{\R}$ une fibration en surfaces quadriques d'espace total lisse. Supposons que $X$ n'est pas de type (I). 
 S'il y a une section alors la vari\'et\'e $X$ est rationnelle sur $\R$.
 Supposons qu'il n'y a pas de section.
 Soit $K=\R(\P^1)$.
 L'application $$\Br(K) \to \Br(\R(X))$$ est injective si et seulement si
 le discriminant  de la forme quadratique d\'efinissant la fibre g\'en\'erique n'est pas un carr\'e;
 sinon son noyau est $\Z/2$ \cite[Prop. 7.2.4]{CTSk}.
 Pour $P \in \P^1(\R)$ on note $K_P$ le compl\'et\'e 
 de $K$ au point $P$. On note $t_{P} \in K_{P}$ une uniformisante.
 Soit  $T$ l'ensemble des points   $P \in \P^1(\R)$
 tels que  $X\times_{K}K_{P}$ soit 
 $K_{P}$-birationnel 
 \`a une quadrique d\'efinie
 par  une \'equation
 $$ (x^2+y^2) -t_{P} (u^2+v^2) = 0.$$
 
 Le groupe  $H^2_{nr}(\R(X)/\R, \Z/2)$
   est engendr\'e par les \'el\'ements de $\Br(\R(\P^1))$
   dont tous les r\'esidus aux points hors de $T$ sont nuls  (voir  \cite{Sk} et 
    \cite[Theorem 2.3.1]{CTSD}):
   \begin{itemize}
\item [(*)]  Si le discriminant n'est pas un carr\'e,
   le groupe $H^2_{nr}(\R(X)/\R, \Z/2)$ n'est pas r\'eduit \`a $\Br(\R)$ si le cardinal de $T$
   est au moins 2.
 \item [(**)] Si le discriminant est un carr\'e, le groupe $H^2_{nr}(\R(X)/\R, \Z/2)$ n'est pas r\'eduit
   \`a $\Br(\R)$ si le cardinal de $T$ est au moins 4. 
   \end{itemize}
     
   On peut ainsi donner des exemples de fibrations en surfaces quadriques 
   $X \to \P^1_{\R}$ telles que $X(\R)$ soit connexe mais que $X$
   ne soit pas stablement rationnelle, ni m\^eme $CH_{0}$-triviale.
   
On consid\`ere   la famille donn\'ee dans $\P^3_{\R} \times_{\R} \A^1_{\R}$
avec coordonn\'ees $(x,y,z,t;u)$
  par l'\'equation
  $$ x^2+(1+u^2) y^2 - u (z^2+t^2)=0.$$
  On peut recoller cette famille avec la famille donn\'ee par l'\'equation
  $$x'^2+(1+v^2)y'^2-v(z'^2+t'^2)=0$$
  dans $\P^3_{\R} \times_{\R} \A^1_{\R}$
  avec coordonn\'ees  $(x',y',z',t'; v)$, au moyen du
 changement de variable $(x',y',z',t';v)=(vx, y,z,t; 1/u)$.
  Ceci donne un mod\`ele admissible au sens de \cite[\S 2]{Sk} et \cite[\S 3]{CTSk}.
 On  en d\'eduit (voir  \cite[Prop. 2.4]{Sk}, \cite[Thm. 3.3]{CTSk}) une construction explicite
 d'un mod\`ele $\pi : X \to \P^1_{\R}$ avec $X/\R$ projectif et lisse.
  Pour $u \in \R$, les fibres lisses $X_{u}$ satisfont $X_{u}(\R)\neq \emptyset$ 
  (et connexe) si et seulement si $u >0$.
  L'espace $X(\R)$ 
 a  donc  exactement une composante connexe.
 
   \begin{prop}\label{refX}
  Soit $\pi : X \to \P^1_{\R}$ la fibration en surfaces quadriques d\'efinie ci-dessus: c'est un mod\`ele projectif et lisse de la vari\'et\'e
  $$ x^2+(1+u^2) y^2 - u (z^2+t^2)=0\subset \P^3_{\R} \times_{\R} \A^1_{\R}.$$
  Alors l'image $\alpha$ de la classe $(-1,u) \in H^2(\R(u), \Z/2)$ dans $H^2(\R(X),\Z/2)$ n'est pas dans l'image de $\Br(\R)$, et  on a $$\alpha\in \Br(X)[2]\subset H^2(\R(X),\Z/2).$$
Ainsi l'application  
  $$\Br(\R) \to \Br(X)$$
  n'est pas surjective.
\end{prop}
\begin{proof}
 L'application de restriction $\Br(\R(u)) \to \Br(\R(X))$ est injective  car le discriminant $u^2(1+u^2)$ 
 de la forme quadratique $x^2+(1+u^2) y^2 - u (z^2+t^2)$
 n'est pas un carr\'e dans $\R(u)$
 \cite[Prop. 7.2.4]{CTSkB}.
 La classe $(-1,u) \in H^2(\R(u), \Z/2)$ ne vient pas de $\Br(\R)$ 
 car son r\'esidu en $u=0$ n'est pas nul. 
 La classe $\alpha \in \Br(\R(X))$ n'est donc pas dans l'image de $\Br(\R)$.
Puisque $X$ est lisse sur $\R$, pour \'etablir l'\'enonc\'e 
 il suffit de montrer que $\alpha$ est non ramifi\'ee sur $\R(X)$ \cite{Santa}. Soit  
   $A \subset \R(X)$   un anneau de valuation discr\`ete, $w$ la valuation.
   Si  la valuation $w(u)$ est paire, alors  le r\'esidu de $(-1,u)$ est trivial.
Si la valuation $w(u)$ est impaire,  on v\'erifie sur l'\'equation  
que $(-1)$ est un carr\'e dans le corps r\'esiduel  de $A$, et donc le r\'esidu de $(-1,u)$ est trivial.
  \end{proof} 
 
En particulier, $X$ n'est pas universellement $CH_{0}$-triviale.
Remarquons que $\alpha \in \Br(X)$  n'est pas dans l'image de $\Br(\R)$ et en particulier est non nul,
 mais s'annule en tout point de $X(\R)$.

 \subsection{Hypersurface cubique singuli\`ere dans $\P^4_{\R}$}

 La vari\'et\'e de la proposition \ref{refX} est $\R$-birationnelle \`a l'hypersurface cubique
de $\A^4_{\R}$
d'\'equation affine
$$ x^2+ (1+u^2)   - u (z^2+t^2)=0.$$
Celle-ci admet un mod\`ele projectif  $Y \subset \P^4_{\R}$ d'\'equation
$$ x^2.y + y.(y^2+u^2) - u(z^2+t^2)=0$$
en coordonn\'ees homog\`enes $(x,y,u,z,t)$.
Le lieu singulier est form\'e des 4 points
non r\'eels donn\'es par 
$$(x,y,u,z,t) = (\pm i, 0,1,0,0), (0,0,0, \pm i, 1).$$
Ceci donne un exemple d'hypersurface cubique  $Y$ singuli\`ere dans $\P^4_{\R}$
avec un mod\`ele projectif et lisse $Y'$ satisfaisant $Y'(\R)$ connexe
et $\Br(Y')/\Br(\R) \neq 0$, donc $Y'$ non stablement rationnelle, ni m\^{e}me
universellement $CH_{0}$-triviale.

\medskip
Par comparaison,
la vari\'et\'e affine d'\'equation $x^2+y^2+z^2-u(u^2+1)=0$ sur $\R$ admet une compactification 
\'evidente comme 
hypersurface cubique $Y$ d'\'equation homog\`ene
$$ (x^2+y^2+z^2)t -u (u^2+t^2)=0.$$
Les points singuliers sont donn\'es par $u=t=0=x^2+y^2+z^2$.
C'est une conique plane sans point r\'eel.
Il n'y a pas de point r\'eel singulier.  
Nous avons vu \`a la section  \ref {scomp} (deuxi\`eme exemple) que tout mod\`ele projectif et lisse de $Y$
est universellement $CH_{0}$-trivial.

\subsection{Intersection singuli\`ere de deux quadriques dans $\P^5_{\R}$}
 
La vari\'et\'e de la proposition \ref{refX} est aussi $\R$-birationnelle \`a l'intersection de
deux quadriques  $W$ dans $\P^5_{\R}$
donn\'ee par les \'equations
 $$ x^2+ y^2+u^2   - u w=0.$$
$$ wy= z^2+t^2.$$
Le lieu singulier est form\'e des 4 points
non r\'eels donn\'es par 
$$(x,y,u,w,z,t)= (\pm i, 0, 1,0,0,0)$$
 $$(x,y,u,w,z,t)=(0,0,0,0,\pm i, 1).$$
 La section par $w=0$ est donn\'ee par 
 $x^2+y^2+u^2=0, z^2+t^2=0$. En particulier
 elle ne contient pas de point r\'eel. Ceci implique que
 $W$ ne contient pas de droite r\'eelle.
 Ceci donne un exemple d'intersection de deux quadriques  $Y$ dans $\P^5_{\R}$
avec un mod\`ele projectif et lisse $Y'$ satisfaisant $Y'(\R)$ connexe
et $\Br(Y')/\Br(\R) \neq 0$, donc $Y'$ non stablement rationnelle, ni m\^{e}me
universellement
$CH_{0}$-triviale.

\medskip
Par comparaison,
  en posant $w=u^2+1$,
   on peut voir la vari\'et\'e affine  d'\'equation  $$x^2+y^2+z^2-u(u^2+1)=0$$ 
   dont on a rappel\'e ci-dessus que tout mod\`ele projectif et lisse
   est universellement $CH_{0}$-trivial,
   comme
intersection 
de deux quadriques affines 
  $$x^2+y^2+z^2-uw=0,$$
  $$w=u^2+1.$$
  La compactification \'evidente   $Y \subset \P^5_{\R}$ donn\'ee par
$$x^2+y^2+z^2-uw=0,$$
  $$ u^2+t^2 -wt=0$$
  est singuli\`ere.
  Le lieu singulier est donn\'e par $$u=t=w=0=x^2+y^2+z^2,$$
  c'est une conique plane sans point r\'eel.
   La trace de l'hyperplan $w=0$ sur $Y(\R)$ est vide.
  Ainsi $Y$ ne contient aucune droite r\'eelle de $\P^5_{\R}$.

\subsection{La m\'ethode des jacobiennes interm\'ediaires}\label{IJTWitt}

Olivier Wittenberg nous informe que la m\'ethode des jacobiennes interm\'ediaires comme \'etablie dans
\cite{BW21}  donne le r\'esultat suivant.

\begin{thm}\label{wittenberg23}
Soit $k$ un corps de caract\'eristique diff\'erente de $2$. Soit $\pi : X \to \P^1_{k}$ une fibration en surfaces quadriques de type (I), avec au moins 6 fibres g\'eom\'etriques singuli\`eres. Soit $\Delta \to \P^1_{k}$ le rev\^etement double associ\'e, et soit $\beta \in \Br(\Delta)$ la classe associ\'ee \`a cette fibration. 
Supposons $\Delta(k) \neq \emptyset$.
Si $\beta$ n'est pas dans l'image de $\Br(k)$, alors la $k$-vari\'et\'e $X$ n'est pas $k$-rationnelle.
\end{thm}

Dans le cas $k=\R$, cela lui permet de donner de nombreux exemples
de fibrations de type (I) telles que $X(\R)$ est connexe mais $X$ n'est pas
$\R$-rationnelle. Pour les fibrations (\ref{eq})  consid\'er\'ees dans le pr\'esent article, 
la m\'ethode ne s'applique pas :
 la classe $\beta$  est non nulle mais dans l'image de $\Br(k)$.
De plus, si le polyn\^{o}me $P(u) \in \R[u]$ est de degr\'e 3,  la fibration $X \to \P^1_{\R}$ a seulement 4
fibres g\'eom\'etriques singuli\`eres.

\section{Intersections lisses de deux quadriques dans $\P^5_{k}$}\label{compdeuxquad}

Soit $Y \subset \P^5_{k}$ une intersection compl\`ete lisse de deux quadriques
d\'efinies par un syst\`eme $f=g=0$.
Supposons que $Y$ poss\`ede un $k$-point $M$. 
Soit $H \subset \P^5_{k}$
l'espace tangent \`a $Y$ en $M$. La famille des espaces lin\'eaires $\P^4_{k} \subset \P^5_{k}$
contenant $H$ d\'efinit une application rationnelle 
de $Y$ vers $\P^1_{k}$.  Dans \cite[\S  5.2]{CT23}, on montre qu'on obtient ainsi une $k$-vari\'et\'e projective et lisse $$X \subset \P^3_{k} \times_{k} \P^1_{k}$$
qui est $k$-birationnelle \`a $Y$. La fibration $\pi : X \to \P^1_{k}$
a pour fibres des quadriques.  
Les fibres g\'eom\'etriques  singuli\`eres sont d\'efinies par des formes de rang 3.
Il y a exactement 6 fibres g\'eom\'etriques singuli\`eres, correspondant aux z\'eros de la
forme homog\`ene $det(\lambda f+ \mu g)$.
 La fibration $X \to \P^1_{k}$ est de type (I).
 Le  rev\^{e}tement double associ\'e $\Delta \to \P^1_{k}$
 est donn\'e par $$z^2=- det (\lambda f+\mu g).$$

Comme $Y$ poss\`ede un point rationnel, la forme quadratique g\'en\'erique
$f+tg$ poss\`ede un z\'ero sur le corps $k(t)$, elle admet donc une
d\'ecomposition
$$ f+t g \simeq  <1,-1> \perp \Phi$$
avec $\Phi$ de rang 4 sur $k(t)=k(\P^1)$.
La fibre g\'en\'erique de la fibration $X \to \P^1_{k}$ 
est d\'efinie par $\Phi=0$.

La classe $\beta \in\Br(\Delta)$ est la classe associ\'ee \`a la forme
quadratique  $\Phi_{k(\Delta)}$,  qui est de rang 4, et de discriminant  
un carr\'e dans $k(\Delta)^{\times}$.

En utilisant \cite[Prop. 5.6]{CT23},  on v\'erifie, sous l'hypoth\`ese $Y(k) \neq \emptyset$, que
la classe $\beta \in \Br(k(\Delta))$ ici d\'efinie co\"{\i}ncide 
avec la classe $\alpha:= \alpha_{Y} \in \Br(k(\Delta))$
 associ\'ee \`a $Y$ dans \cite[\S 5.2]{CT23}.

La proposition suivante \'etend  \cite[thm. 5.10]{CT23}. Elle nous
a \'et\'e communiqu\'ee  par   Olivier Wittenberg.

\begin{prop}\label{510fin}
Soit $Y \subset \P^5_{k}$ une intersection lisse de deux quadriques donn\'ee par
un syst\`eme $f=g=0$.  Les propri\'et\'es suivantes sont \'equivalentes :
\begin{itemize}
\item [(i)] la $k$-vari\'et\'e $Y$ contient une droite sur $k$;
\item [(ii)] on a $\alpha_{Y}=0 \in \Br(k(\Delta))$;
\item [(iii)] la classe $\alpha_{Y} \in \Br(k(\Delta))$  est dans l'image de $\Br(k) \to \Br(k(\Delta))$.
\end{itemize}
Elles impliquent :
\begin{itemize}
\item[(iv)] la $k$-vari\'et\'e $Y$ est $k$-rationnelle.
\end{itemize}
 \end{prop}
\begin{proof}
Nous n'indiquons ici que le compl\'ement de d\'emonstration.
Supposons que la classe $\alpha$ est l'image de $\gamma \in \Br(k)$. Soit $Z$
une $k$-vari\'et\'e de Severi-Brauer de classe $\gamma \in \Br(k)$.
Sur le corps des fonctions $k(Z)$, la classe $\beta$ s'annule.
Soit $F_{1}(Y)$ la $k$-vari\'et\'e des droites de $Y$.
C'est un espace principal homog\`ene sous une $k$-vari\'et\'e ab\'elienne $J$.
D'apr\`es \cite[thm. 5.10]{CT23}, la classe de $F_{1}(Y)$
dans $H^1(k,J)$ s'annule alors dans $H^1(k(Z),J)$.
Cela signifie qu'il y a une $k$-application rationnelle
de $Z$ vers $F_{1}(Y)$. Mais toute  application rationnelle
d'une  vari\'et\'e g\'eom\'etriquement rationnelle vers
un espace homog\`ene d'une
vari\'et\'e ab\'elienne est constante.
Ainsi $F_{1}(Y)(k) \neq \emptyset$, donc $Y \subset \P^5_{k}$ 
contient une droite d\'efinie sur $k$. 
D'apr\`es \cite[thm. 5.10]{CT23}, ceci 
implique $\alpha=0 \in \Br(k(\Delta))$.
\end{proof}

\begin{rmk}
Supposons que $Y$ poss\`ede un point rationnel.
On consid\`ere une fibration $\pi : X \to \P^1_{k}$ associ\'ee.
Si l'on a $(ii)$, on peut aussi d\'eduire $(iv)$ de la mani\`ere suivante.
La fibration en coniques sur $\Delta$ associ\'ee est alors triviale.
La fibre g\'en\'erique de $\pi$ est la descendue \`a la Weil
pour l'extension $k(\Delta)/k(\P^1)$ 
de $\P^1_{k(\Delta)}$. Elle est donc birationnelle \`a $\P^2_{k(\P^1)}$
et donc $X$ est $k$-rationnelle. De fait il y a \'equivalence entre
le fait que la quadrique  fibre g\'en\'erique de $\pi : X \to \P^1_{k}$ a un $k(\P^1)$-point,
ou encore qu'elle est $k(\P^1)$-rationnelle, et   l'\'egalit\'e $\alpha_{X}=0 \in \Br(k(\Delta))$.
\end{rmk}

 \begin{rmk}
 C'est  un th\'eor\`eme d'O. Benoist et O. Wittenberg  \cite{BW21} que $(iv)$ implique
 les autres \'enonc\'es. Ceci n'est pas utilis\'e dans le corollaire suivant.
 
 \end{rmk}

\begin{cor}
Soit $\pi : X \to \P^1_{k}$ une  fibration en surfaces quadriques donn\'ee
par une \'equation affine
$$ x^2-ay^2-bz^2 = u.p(u)$$
avec $p(u)$ s\'eparable non constant et $p(0)\neq 0$, et avec  $(a,b)\neq 0 \in \Br(k)$.
 La fibration $\pi : X \to \P^1_{k}$
n'est pas birationnellement \'equivalente sur $\P^1_{k}$ \`a une fibration associ\'ee \`a
 une intersection compl\`ete lisse de deux
quadriques dans $\P^5_{k}$ poss\'edant un point rationnel.
 \end{cor}

\begin{proof} Comme le polyn\^{o}me $u.p(u)$ a un z\'ero, 
la courbe $\Delta$ poss\`ede un point rationnel. On a donc
$\alpha_{X/\P^1}=(a,b) \neq 0 \in \Br(\Delta) \subset \Br(k(\Delta))$. 
L'\'equivalence entre $(ii)$ et $(iii)$ dans la proposition \ref{510fin}
permet de conclure.
\end{proof}

\section{Corps de nombres et corps $p$-adiques}\label{corpsdenombres}

Soit $X\to \P^1_{\R}$ la fibration en surfaces quadriques donn\'ee par l'\'equation affine
$$x^2+y^2+z^2=u(u^2+1).$$

Le polyn\^ome $p(u)=u^2+1$ est une somme de 2 carr\'es dans $\Q[u]$.  D'apr\`es la preuve du th\'eor\`eme \ref{ex2}, la fonction $r(u,v)$ est une somme de 3 carr\'es dans $\Q(\sqrt{3})(u,v)$. D'apr\`es le th\'eor\`eme \ref{equivalencesk}, on a donc que pour $k=\Q(\sqrt{3})$ la $k$-vari\'et\'e $X$ d'\'equation affine $x^2+y^2+z^2=u(u^2+1)$ est universellement $CH_0$-triviale.

En g\'en\'eral, on ne peut pas esp\'erer \'etendre les r\'esultats  de $CH_{0}$-trivialit\'e sur $\R$ aux corps de nombres et corps $p$-adiques.	
	Soit $k$ un corps de caract\'eristique diff\'erente de $2$. Soit $X$ une $k$-vari\'et\'e projective
	et lisse munie d'un morphisme plat $\pi : X \to \P^1_{k}$ dont la fibre g\'en\'erique
	est une surface quadrique lisse. Supposons $X(k) \neq \emptyset$.
	Pour tout surcorps $L$ de $k$, le groupe $A_{0}(X)$ est d'exposant~2.
	Soit $k$ un corps de nombres. Notons $k_{v}$ le  compl\'et\'e de $k$
	en une place $v$ et $X_{v}:=X\times_{k}k_{v}$.
	Pour $v$ place complexe, $A_{0}(X_{v})=0$. Pour $v$ place r\'eelle, on a $A_{0}(X_{v})=(\Z/2)^{s_{v} -1}$,
	o\`u $s_{v}$ est le nombre de composantes connexes de l'espace topologique
	$X_{v}(k_{v})$ \cite{CTIschebeck}.
	 Pour presque toute place finie $v$, on a $A_{0}(X_{v})=0$ (\cite[Thm. 6.2 (vi)]{CTSk} et \cite[Thm. 5.3]{PS95}).
	  
        On a un complexe naturel de groupes ab\'eliens finis
	$$ A_{0}(X) \to \oplus A_{0}(X_{k_{v}}) \to {\rm Hom(}\Br(X)/\Br(k), \Q/\Z),$$
o\`u les accouplements
$$A_{0}(X_{k_{v}}) \times \Br(X)/\Br(k)  \to \Q/\Z$$
 sont donn\'es par l'\'evaluation. Pour $X$ une fibration en quadriques sur $\P^1_{k}$,
 ce complexe est une suite exacte (Wittenberg, \cite[Cor. 1.7]{W12}).

 \begin{thm} Soient $k$ un corps de nombres et $\pi : X \to \P^1_{k}$ une fibration en surfaces quadriques de type (I) 
 satisfaisant $X(k)\neq \emptyset$.
 On a un isomorphisme de groupes ab\'eliens finis de 2-torsion
 $$ A_{0}(X) \simeq \oplus_{v} A_{0}(X_{v}).$$
 \end{thm}

\begin{proof}
 Pour une fibration de type (I), on a  $\Br(X)/\Br(k)=0$. L'application
 diagonale  $$A_{0}(X) \to \oplus_{v} A_{0}(X_{k_{v}}) $$
  est donc surjective.   
  Montrons qu'elle est injective.
  Avec les notations de la
   section \ref{rappelsCTSk},
on a les injections
$$A_{0}(X) \hookrightarrow k(\Delta)_{dn}^{\times}/k^{\times}.N_{D}(k(\Delta))$$
$$A_{0}(X_{v}) \hookrightarrow
 k_{v}(\Delta)_{dn}^{\times}/
 k_{v}^{\times}. N_{D}(k_{v}(\Delta)).$$
 On a les injections
$$k(\Delta)^{\times}/ N_{D}(k(\Delta)) \hookrightarrow H^3(k(\Delta), \Z/2)$$
$$k_{v}(\Delta)^{\times}/ N_{D}(k(\Delta_{v})) \hookrightarrow H^3(k_{v}(\Delta), \Z/2).$$
On a donc les injections
$$k(\Delta)^{\times}_{dn}/ N_{D}(k(\Delta)) \hookrightarrow H^3_{nr}(k(\Delta), \Z/2)$$
$$ k_{v}(\Delta)_{dn}^{\times}/N_{D}(k(\Delta_{v})) \hookrightarrow H^3_{nr}(k_{v}(\Delta), \Z/2).$$
 D'apr\`es Kato  \cite{Kato}, pour presque toute place $v$,  le groupe  $H^3_{nr}(k_{v}(\Delta), \Z/2)$
est nul.   Il en est  de m\^{e}me  pour les groupes $A_{0} (X_{v} )$.
Soit $S$ l'ensemble fini des places $v$ avec  $A_{0}(X_{v}) \neq 0$.
 D'apr\`es Kato  \cite{Kato}
 l'application diagonale
$$H^3(k(\Delta), \Z/2) \to \prod_{v} H^3(k_{v}(\Delta), \Z/2)$$
est injective. Il est en donc de m\^eme de
$$H^3_{nr}(k(\Delta), \Z/2) \to \prod_{v\in S} H^3_{nr}(k_{v}(\Delta), \Z/2).$$
On a des suites exactes \'evidentes et compatibles avec les applications diagonales :
$$ k^{\times} \to k(\Delta)_{nr}^{\times}/N_{D}(k(\Delta)) 
\to k(\Delta)_{nr}^{\times}/k^{\times}. N_{D}(k(\Delta)) \to 1$$
$$ \prod_{v \in S}k_{v}^{\times} \to 
 \prod_{v \in S} k_{v}(\Delta)_{nr}^{\times}/N_{D}(k(\Delta_{v})) 
\to  \prod_{v \in S} k_{v}(\Delta)_{nr}^{\times}/k_{v}^{\times}. N_{D}(k(\Delta_{v})) 
\to 1.$$

Soit $\xi \in A_{0}(X)$ d'image  $\xi_{0} \in k(\Delta)_{nr}^{\times}/k^{\times}.N_{D}(k(\Delta))$
et d'image nulle dans tous les $A_{0}(X_{v})$.
Relevons $\xi_{0}$ en $\xi_{1} \in k(\Delta)_{nr}^{\times}/N_{D}(k(\Delta))$.
Comme l'image de $\xi$ dans les groupes $A_{0}(X_{v})$ est nulle,
pour chaque $v\in S$, il existe $\lambda_{v} \in k_{v}^{\times}$ 
d'image $\xi_{1,v}  \in k_{v}(\Delta)_{nr}^{\times}/N_{D}(k(\Delta))$.
Par approximation faible aux places de $S$,
on trouve $\lambda \in k^{\times}$ tel que $\xi_{0}/\lambda \in 
k(\Delta)_{nr}^{\times}/N_{D}(k(\Delta))$ ait une image nulle par passage aux
compl\'et\'es $k_{v}$ pour $v\in S$. Ceci implique 
$$\xi_{0}/\lambda =1  \in 
k(\Delta)_{nr}^{\times}/N_{D}(k(\Delta))$$ et donc $\xi=0 \in A_{0}(X)$.
   \end{proof}
 
 Ainsi, s'il existe une place $v$ telle que le groupe  $A_{0}(X_{k_{v}})$ n'est pas trivial, alors  $A_0(X)\neq 0$. En particulier, c'est le cas dans l'exemple qui suit.

Dans \cite[Proposition 6.2, p.108]{PS95}
 Parimala et
Suresh montrent  que pour la famille
$X\to \P^1_{\Q}$
d'\'equation   affine
$$x^2+y^2+3z^2= 3u(u+2)(u+3)$$
le groupe 
$A_0(X_{\Q_{3}})$ 
n'est pas nul. 
 En multipliant cette \'equation par $9$ et en faisant un changement de variables, on voit que  la  $\Q$-vari\'et\'e $X$ est isomorphe \`a la vari\'et\'e d'\'equation affine 
$$x^2+y^2+3z^2= u(u+6)(u+9).$$

Par approximation faible,  on trouve un polyn\^{o}me $p(u) \in \Q[u]$ de degr\'e 2
proche de  $(u+6)(u+9)\in  \Q_{3}[u]$ et de $u^2+1 \in \R[u]$.
On obtient ainsi un exemple de fibr\'e en surfaces quadriques 
 $X \to \P^1_{\Q}$  de type (I)   d'\'equation affine
 $$x^2+y^2+3z^2= u.p(u)$$
 avec $A_{0}(X_{\R})=0$, 
 $A_{0}(X_{\Q_{3}   })\neq 0$ 
 et $A_{0}(X) \neq 0$.

Notons que dans cet exemple sur $\Q_{3}$
 la condition $$((u+6)(u+9),-1,-3)=0 $$ du th\'eor\`eme \ref{equivalencesk}
n'est pas satisfaite, comme on voit par r\'esidu
$(-1,-3)_{\Q_3}  \neq 0.$


\begin{thebibliography}{CTSS83}

\bibitem[Ar75]{Ar} J\'{o}n Kr. Arason,  Cohomologische Invarianten quadratischer Formen
J. Algebra {\bf 36} (1975), no. 3, 448--491.

 

\bibitem[ACTP17]{ACTP} A. Auel, J.-L. Colliot-Th\'el\`ene, R. Parimala, Universal unramified cohomology  of cubic fourfolds containing a plane, in {\it Brauer groups and obstruction problems} Progr. Math. {\bf 320} Birkh\"auser/Springer, Cham, 2017, p. 29--55.

\bibitem[BW20]{BW} O. Benoist and O. Wittenberg, 
The Clemens-Griffiths method over non-closed fields. 
Algebr. Geom. {\bf 7} (2020), no.6, 696--721.

\bibitem[BW20a]{BW20a} O. Benoist and O. Wittenberg, On the integral Hodge conjecture for real varieties, I.
Invent. math. {\bf 222} (2020), no.1, 1--77. 

\bibitem[BW20b]{BW20b} O. Benoist and O. Wittenberg, On the integral Hodge conjecture for real varieties, II.
Journal de l'\'Ecole Polytechnique {\bf 7} (2020) 373--429.

\bibitem[BW21]{BW21} O. Benoist and O. Wittenberg, 
Intermediate Jacobians and rationality over arbitrary fields,
Ann. Sc. \'Ec. Norm. Sup\'er., (4) {\bf 56} (2023), no.4, 1029--1084

 

\bibitem[CLRR79]{CLRR} M. D. Choi, T. Y. Lam, B. Reznick, A. Rosenberg, Sums of squares in some integral domains,
J. of Algebra {\bf 65} (1980) 234--256.
 

\bibitem[CT95]{Santa} J.-L. Colliot-Th\'el\`ene, Birational invariants, purity and the Gersten conjecture, K-theory and algebraic geometry: connections with quadratic forms and division algebras (Santa Barbara, CA, 1992), 1--64.
Proc. Sympos. Pure Math., {\bf 58}, Part 1
American Mathematical Society, Providence, RI, 1995.


\bibitem[CT19]{CTter} J.-L. Colliot-Th\'el\`ene, 
Non rationalit\'e stable sur les corps quelconques,  in Birational Geometry of Hypersurfaces, Gargnano del Garda, Italie, 2018, A. Hochenegger, M. Lehn,  P. Stellari ed., Lecture Notes of the Unione Matematica Italiana,  Springer LNM 2019.

\bibitem[CT23]{CT23} J.-L. Colliot-Th\'el\`ene,
  Retour sur l'arithm\'etique des intersections de deux quadriques, 
avec un appendice par A. Kuznetsov, J. reine angew. Math. {\bf 806} (2024), 147--185.

\bibitem[CTI81]{CTIschebeck}  J.-L. Colliot-Th\'el\`ene et F. Ischebeck, L'\'equivalence rationnelle sur les  cycles de dimension z\'ero des vari\'et\'es alg\'ebriques r\'eelles, C. R. Acad. Sc. Paris t. {\bf 292}, S\'erie I (1981) 723--725.


\bibitem[CTPa90]{CTP}  J.-L. Colliot-Th\'el\`ene and R. Parimala, 
  Real components of algebraic varieties and \'etale
cohomology, Invent. math. {\bf 101} (1) 8--99, 1990.

  \bibitem[CTP16]{CTPENS} J.-L. Colliot-Th\'el\`ene et A. Pirutka,
Hypersurfaces quartiques de dimension 3: non-rationalité stable, Ann. Sci. Éc. Norm. Supér. (4) {\bf 49} (2016), no.2, 371--397.

\bibitem[CTSk93]{CTSk}   J.-L. Colliot-Th\'el\`ene et
 A.N. Skorobogatov, Groupes de Chow des z\'ero-cycles des fibr\'es en quadriques, Journal of K-theory {\bf 7} (1993) 477--500.
 
  \bibitem[CTSD94]{CTSD} J.-L. Colliot-Th\'el\`ene and Sir Peter Swinnerton-Dyer, Hasse principle and weak approximation for pencils of Severi-Brauer and similar varieties, J. reine angew. Math. {\bf 453} (1994) 49--112.
 

 
 \bibitem[CTSk21]{CTSkB} J.-L. Colliot-Th\'el\`ene and
 A.N. Skorobogatov, The Brauer--Grothendieck group, Ergeb. Math. Grenzgeb. (3), {\bf 71} [Results in Mathematics and Related Areas. 3rd Series. A Series of Modern Surveys in Mathematic]
Springer, Cham, 2021.
 
 
  \bibitem[Co12]{Com} A. Comessatti,
Fondamenti per la geometria sopra le superficie razionali dal punto di vista reale, Math. Ann. {\bf 73} (1912), no.1, 1--72.
 
  \bibitem[Ek90]{Ek}  
  T. Ekedahl,  An effective version of Hilbert's irreducibility theorem, S\'eminaire de th\'eorie des nombres de Paris 1988-1989,  Progress in Mathematics, Vol. {\bf 91}, Birkh\"{a}user, 1990, 
  241--249.
  
\bibitem[FJSVV22]{FJSVV} S. Frei, L. Ji, S. Sankar, B. Viray, and I. Vogt, Curve classes on conic bundle threefolds and applications to rationality, \`a para\^itre dans Algebraic Geom.

\bibitem[G64]{G64} W.-D. Geyer, Ein algebraischer Beweis des Satzes von Weichhold \"{u}ber reelle algebraische Funktionenk\"orper, in {\it Algebraische Zahlentheorie}, Herausgegeben von  H. Hasse und P. Roquette,
Berichte aus dem Mathematischen Forschungsinstitut Oberwolfach, Heft {\bf 2}, Bibliographisches Institut, Mannheim,
Hochschultaschenb\"ucher-Verlag, 1966,   S. 83--98.

\bibitem[GH81]{GH81} B. H. Gross and J. Harris, Real algebraic curves, Ann. Sc. \'Ec. Norm. Sup\'er., 4\`eme s\'erie
t. {\bf 14}, no. 2 (1981) 157--182.
 
 \bibitem[HT21]{HT} B. Hassett and Yu. Tschinkel, Rationality of complete intersections of two quadrics over nonclosed fields, 
 L'Enseignement math\'ematique {\bf 67} (2021) no. 1--2, 1--44. With an appendix by Jean-Louis Colliot-Th\'el\`ene.
 
\bibitem[JJ24]{JJ} L. Ji and M. Ji, Rationality of real conic bundles with quartic discriminant curve, Int. Math. Res. Not.  {\bf 1} (2024), 115--151.
 
 

\bibitem[KRS98]{KRS} B Kahn, M. Rost, R. Sujatha, 
Unramified cohomology of quadrics. I, Amer. J.
Math. {\bf 120} (1998), no. 4, 841--891.   

\bibitem[K86]{Kato}  K. Kato,
A Hasse principle for two-dimensional global fields,
J. reine angew. Math. {\bf 366} (1986) 142--181.


\bibitem[Lam00]{Lam} T.Y. Lam, 
Introduction to Quadratic Forms over Fields, Graduate Studies in Mathematics, Volume {\bf 67}, Amer. Math. Soc. (2000).
 

\bibitem[LMFDB]{LMFDB} The LMFDB Collaboration, The L-functions and modular forms database, https://www.lmfdb.org, 2024, [Online; accessed 8 April 2024].

\bibitem[MS82]{MS} A. S. Merkur'ev and A. A. Suslin,
 {\cyrit{$K$-kogomologii mnogoobrazi\u{i} Severi-Brau\`{e}ra i gomomorfizm normennogo vycheta}} 
($K$-cohomology of Severi-Brauer varieties and the norm residue homomorphism),
Izv. Akad. Nauk SSSR Ser. Mat. \textbf{46} (1982), no. 5, 1011--1046, 1135--1136.

\bibitem[Mer08]{Merk} A. S. Merkurjev, Unramified elements in cycle modules, J. London Math. Soc. (2) {\bf 78} (2008) no. 1, 51--64.

\bibitem[PS95]{PS95} R. Parimala and V. Suresh, Zero-cycles on quadric fibrations : Finiteness and the cycle map,  Inv. math. {\bf 122} (1995) 83--117.

\bibitem[PS16]{PS16} R. Parimala and V. Suresh,
 Degree 3 cohomology of function fields of surfaces
Int. Math. Res. Not. IMRN 2016, no. 14, 4341--4374.

\bibitem[R96]{Rost} M. Rost, 
Chow groups with coefficients, Doc. Math. {\bf 1} (1996), No. 16, 319--393.

 \bibitem[Sil89]{Silhol} R. Silhol, Real Algebraic Surfaces, Springer LNM {\bf 1392} (1989).


\bibitem[S92]{S92} J. H. Silverman, The arithmetic of elliptic curves, Grad. Texts in Math., {\bf 106} Springer-Verlag, New York, 1992.

\bibitem[S94]{S94} J. H. Silverman,
Advanced topics in the arithmetic of elliptic curves, 
Grad. Texts in Math., {\bf 151} Springer-Verlag, New York, 1994.

\bibitem[Sk90]{Sk} A. N. Skorobogatov, Arithmetic on certain quadric bundles  of relative dimension 2. I. 
J. reine angew. Math. {\bf 407} (1990), 57--74.

\bibitem[W48]{W48} A. Weil, Vari\'et\'es ab\'eliennes et courbes alg\'ebriques, 
Actualit\'es scientifiques et industrielles no. 1064, Hermann, Paris (1948).

\bibitem[W34]{W1} E. Witt,
Zerlegung reeller algebraischer Funktionen in Quadrate. Schiefk\"{o}rper \"{u}ber reellem
Funktionenk\"{o}rper, J. reine angew. Math. {\bf 171} (1934), 4--11.
Gesammelte Abhandlungen, Springer-Verlag Berlin Heidelberg 1998, 81--88.

\bibitem[W37]{W2} E. Witt,
Theorie der quadratischen Formen in beliebigen K\"{o}rpern, J. reine angew. Math. {\bf 176}
(1937), 31--44. Gesammelte Abhandlungen, Springer-Verlag Berlin Heidelberg 1998, 2--15.

\bibitem[W12]{W12} O. Wittenberg,
 Z\'ero-cycles sur les fibrations au-dessus d'une courbe de genre quelconque,
  Duke Mathematical Journal {\bf 161} no. 11 (2012) 2113--2166.

\bibitem[Z00]{Z00}  Yu. G. Zarhin,
Hyperelliptic jacobians without complex multiplication.
Mathematical Research Letters {\bf 7}, 123--132 (2000).


\bibitem[Z24]{Z24}  Yu. G. Zarhin,
Odd and even elliptic curves with complex multiplication. $\mathtt{arXiv:2406.07240}$, pr\'epublication (2024).


	\end{thebibliography}
 \end{document}